\documentclass[a4paper,11pt,DIV=12]{scrartcl}

\usepackage[utf8]{inputenc}
\usepackage[T1]{fontenc}
\usepackage{lmodern}
\usepackage{amsmath,amssymb,amsthm}
\usepackage{graphicx}
\usepackage{float}
\usepackage[dvipsnames,svgnames]{xcolor}
\usepackage[colorlinks=true,linkcolor=RoyalBlue,citecolor=PineGreen,urlcolor=RoyalBlue]{hyperref}
\usepackage{tikz,pgf}
\usetikzlibrary{arrows.meta}
\usetikzlibrary{calc}
\usetikzlibrary{intersections}

\usepackage{array}
\usepackage{booktabs}
\usepackage{multirow}
\usepackage{hyperref}
\usepackage{caption}
\usepackage{subcaption}
\usepackage[nameinlink]{cleveref}
\usepackage{enumitem}

\newtheorem{theorem}{Theorem}[section]
\newtheorem{lemma}[theorem]{Lemma}
\newtheorem{proposition}[theorem]{Proposition}

\theoremstyle{definition}\newtheorem{definition}[theorem]{Definition}
\theoremstyle{definition}\newtheorem{remark}[theorem]{Remark}

\colorlet{colG}{DarkSeaGreen}
\definecolor{colR}{HTML}{CC6677}
\definecolor{colO}{HTML}{DDCC77}
\definecolor{colB}{HTML}{6699CC}
\colorlet{colY}{Gold!70!black}
\colorlet{colP}{Purple!60!white}
\colorlet{colX}{colR}

\newcommand{\prigid}{\textbf{p}-rigid}
\newcommand{\mrigid}{\textbf{m}-rigid}
\newcommand{\urigid}{\textbf{u}-rigid}

\newcommand{\blue}{\text{blue}}
\newcommand{\red}{\text{red}}

\DeclareRobustCommand{\asc}[1]{\tikz[scale=0.225,baseline={(0,0)}]{\draw[](0, 0)--(1, 0)--(1, 1)--(0, 1)--cycle;\draw[](0, 1)--(1, 0);}\,(#1)}
\newcommand{\sst}[2]{\left\lbrace #1\,\textbf{\upshape\textbar}\,#2 \right\rbrace}

\colorlet{colbg}{white}
\colorlet{colfg}{black}
\colorlet{colgraphv}{colfg!75!white}
\colorlet{colgraphe}{colfg!55!white}
\tikzstyle{vertex}=[fill=colgraphv,circle,inner sep=0pt, minimum size=4pt]
\tikzstyle{hvertex}=[vertex,colP,rectangle,rounded corners=1pt]
\tikzstyle{edge}=[line width=1.5pt,colgraphe]
\tikzstyle{labelsty}=[colgraphv,font=\scriptsize]

\newcommand{\iedge}[1]{\tikz[baseline={(0,0)}]{\draw[edge,#1] (0,0)--(0,0.25);}}
\newcommand{\ivertex}[1]{\tikz[baseline={(0,-0.125)}]{\node[hvertex,#1] at (0,0) {};}}

\pgfdeclarelayer{background}
\pgfsetlayers{background,main}

\begin{document}

\title{On rigidity properties of unit-distance graphs}
\author{%
    Sean Dewar\thanks{Numerical Analysis and Applied Mathematics (NUMA), KU Leuven} 
    \and 
    Georg Grasegger\thanks{Johann Radon Institute for Computational and Applied Mathematics, Austrian Academy of Sciences}
    \and
    Alison La Porta\thanks{Numerical Analysis and Applied Mathematics (NUMA), KU Leuven}
    \and
    Jan Legersk\'y\thanks{Faculty of Information Technology, Czech Technical University in Prague}
    \and
    Anthony Nixon\thanks{School of Mathematical Sciences, Lancaster University, UK}
}
\date{}

\maketitle

\begin{abstract}
    A \emph{unit-distance graph} is a graph which admits a realisation in Euclidean space in which every edge has unit length. Imposing further geometric conditions on the non-edges gives a family of natural subclasses. Requiring that no two vertices lie at distance less than one gives \emph{penny and marble graphs}: the contact graphs of collections of equal radii $d$-dimensional spheres with non-overlapping interiors for $d=2,3$. Requiring instead that the straight-line drawing in the plane be non-crossing gives \emph{matchstick graphs}. Since any motion of a realisation must preserve these extra conditions, the rigidity and flexibility properties of the resulting frameworks differ from those of classical bar-joint rigidity theory. In this note we analyse various rigidity problems for penny and marble graphs, matchstick graphs and unit-distance graphs. In particular we answer a recent open problem on penny and marble graph rigidity and establish a link between penny graphs and the concept of NAC-colourings.
\end{abstract}

\section{Introduction} 

Consider a collection of coins on the table in front of you. Can you change the shape of the collection without breaking the contacts between any pair of coins? To answer this we begin with the model of a \emph{bar-joint framework}: an ordered pair $(G,\rho)$ consisting of a graph $G=(V,E)$ and a realisation $\rho \colon V \to \mathbb{R}^d$ of its vertices into Euclidean $d$-space. The framework is \emph{rigid} if every edge-length-preserving continuous motion of $\rho$ arises from an isometry of $\mathbb{R}^d$, and \emph{globally rigid} if $\rho$ is the unique realisation of its edge lengths up to congruence. 

The rigidity and flexibility properties of such frameworks are well studied \cite{AsimowRoth1978,Connelly2005}
but that of non-generic frameworks less so. In this context it is very natural to consider frameworks with unit edge lengths, or even stronger, as we do in this article, to consider frameworks arising from collections of equal radii $d$-spheres with non-overlapping interiors (such as pennies in the plane). The contact graph of such a collection is a \emph{$d$-sphere graph}, and is called a \emph{penny graph} when $d=2$ and a \emph{marble graph} when $d=3$. The essential point is that the non-overlapping condition is a constraint on the \emph{non}-edges which any motion must also respect: a motion of a penny configuration may not push two disjoint pennies through one another, even though the corresponding bar-joint framework is free to do so. Rigidity within this class is thus a different notion, and one may ask the analogous question for any class of unit-distance realisations with conditions on the non-edges, for example \emph{matchstick graphs}, where the straight-line drawing is required to be non-crossing. 

These questions were investigated in \cite{DewarGraseggerEtAl2025}, where a  characterisation of penny-rigidity and global penny-rigidity was obtained for chordal graphs, along with examples separating the sphere-graph notions from their bar-joint counterparts in both directions.
Two questions left open in \cite{DewarGraseggerEtAl2025} concern instead the relationship at the level of individual graphs, namely whether there exists a generically globally rigid penny graph, or marble graph, which is penny- (respectively marble-) rigid but not globally penny- (respectively marble-) rigid. We answer both questions in \Cref{sec:figs}, in the affirmative, by direct construction; the relevant configurations are given in \Cref{fig:genGlobPennyNonGlob,fig:genGlobMarbleNonGlob}. 

We then turn to two structural results. The first, \Cref{prop:pennyouterplanar}, gives a complete characterisation of rigidity for matchstick graphs that are outerplanar. The second concerns NAC-colourings as introduced in~\cite{GraseggerLegerskySchicho2019}. In \Cref{thm:allFaces3or4cycles} we establish a link between penny-rigidity and the existence of NAC-colourings of a prescribed type.

\section{Rigidity theory preliminaries}
\label{sec:prelims}
We consider finite, simple graphs. For a framework $(G,\rho)$ we call $\rho$ a \emph{realisation}
of the \emph{underlying graph} $G$ in~$\mathbb{R}^d$.
The edges of $G$ are realised as straight line segments between the realisations of its end-vertices. We call the realisations of the vertices and edges of $G$ \emph{joints} and \emph{bars} of $(G,\rho)$ respectively. Realisations $\rho$ and $\rho'$ of $G$ are called \emph{equivalent}, respectively, \emph{congruent} if $\|\rho(u)-\rho(v)\|=\|\rho'(u)-\rho'(v)\|$ for all $uv\in E$, respectively, for all $u,v\in V$.

A \emph{flex} of $(G,\rho)$ is a continuous map $t \mapsto m_t$ from $[0,1]$ to the space
of realisations of $G$ in $\mathbb{R}^d$ such that $m_0(u)=\rho(u)$ for all $u\in V$ and
$m_t$ is equivalent but not congruent to $\rho$ for all $t\in(0,1]$.
If $(G,\rho)$ has no flex, then $(G,\rho)$ is \emph{rigid}.
Otherwise, it is \emph{flexible}.

An \emph{infinitesimal flex} of $(G,\rho)$ is a map $q:V\rightarrow\mathbb{R}^d$ such that
$\langle q(u)-q(v),\rho(u)-\rho(v)\rangle=0$
for all $uv\in E$, but it does not hold for some $u,v\in V$.
If~$(G,\rho)$ has an infinitesimal flex, then $(G,\rho)$ is \emph{infinitesimally flexible}.
Otherwise, $(G,\rho)$ is \emph{infinitesimally rigid}.

An important tool to study infinitesimal rigidity is the \emph{rigidity matrix} of $(G,\rho)$, i.e., an $|E|\times d|V|$ matrix $R(G,\rho)$ such that the edge $uv\in E$ has associated row
\begin{align*}
    \begin{array}{ccccc}
       & u & &  v & \\
    \Big( \dots & [\rho(u)-\rho(v)]^T & \dots & [\rho(v)-\rho(u)]^T &\dots\Big),
    \end{array}
\end{align*}
where all entries not shown are zero.
The kernel of $R(G,\rho)$ contains the infinitesimal flexes of $(G,\rho)$.
If $R(G,\rho)$ has full rank then $(G,\rho)$ is \emph{independent.}

A framework $(G,\rho)$ in $\mathbb{R}^d$ is \emph{globally rigid}
if any other framework $(G,\rho')$ in $\mathbb{R}^d$ with the same corresponding bar-lengths
is congruent to $(G,\rho)$.
If the coordinates of $\rho(V)$ are algebraically independent over $\mathbb{Q}$, then $(G,\rho)$
is \emph{generic}. A classical result asserts that a generic framework $(G,\rho)$ in $\mathbb{R}^d$ is rigid if and only if all generic frameworks $(G,\rho')$ in $\mathbb{R}^d$ are rigid~\cite{AsimowRoth1978}, analogously for the global rigidity~\cite{Connelly2005,GortlerHealyThurston2007}.
Hence, we say $G$ is \emph{generically (globally) $d$-rigid} if some framework $(G,\rho)$ in~$\mathbb{R}^d$ is (globally) rigid.

\section{Frameworks with unit-length bars}
\label{sec:defs}
We call a graph $G=(V,E)$ a \emph{$d$-unit-distance graph} if there exists an injective realisation $\rho\colon V\rightarrow \mathbb{R}^d$ with 
$vw\in E \Rightarrow \|\rho(v)-\rho(w)\|=1$.
Such a realisation~$\rho$ is said to be a \emph{$d$-unit-distance realisation} of $G$ and we call the pair $(G,\rho)$ a \emph{$d$-unit-distance framework}.
Note that in some works, e.g., \cite{AbelDemaineEtAl2025}, injectivity is not required.

A \emph{$d$-sphere graph} is
a graph $G=(V,E)$ such that there exists a $d$-unit distance realisation $\rho\colon V\rightarrow \mathbb{R}^d$ with  
$vw\not\in E \Rightarrow \|\rho(v)-\rho(w)\|>1$.
Such a realisation~$\rho$ is said to be a \emph{$d$-sphere realisation} of~$G$, and we call the pair $(G,\rho)$ a \emph{$d$-sphere framework}. We 
use the terms penny graph/realisation/framework and marble graph/realisation/framework for the 2- and 3-dimensional cases respectively.

When $d=2$, we consider an additional notion: a graph $G$ is a \emph{matchstick graph} if there is a $2$-unit-distance framework $(G,\rho)$ that is planar. Such a realisation $\rho$ is said to be a \emph{matchstick realisation} of $G$ and we call the pair $(G,\rho)$ a \emph{matchstick framework}.

A $d$-unit-distance/$d$-sphere/matchstick graph $G$ is \emph{(globally) $d$-unit-distance}/\emph{$d$-sphere}/\emph{match\-stick rigid}
if there exist finitely many (only one) $d$-unit-distance/$d$-sphere/matchstick realisations of $G$ in $\mathbb{R}^d$ modulo isometries.
Again, we use \emph{(globally) penny-rigid} and \emph{(globally) marble-rigid} to denote (global) $d$-rigidity for the cases where $d=2,3$ respectively. For simplicity, we abbreviate penny-rigid, matchstick rigid and  2-unit-distance-rigid to \prigid, \mrigid\, and \urigid. 

\begin{remark}
    These definitions of (global) rigidity are related to the generic ones as follows.
    Let $\ell:E\rightarrow\mathbb{R}$ be the bar-lengths of a generic framework $(G,\rho')$ with edge set $E$.
    Let $\mathcal{S}_d$ be the set of all realisations~$\rho$ of $G$ in $\mathbb{R}^d$
    up to congruence such that $\|\rho(u)-\rho(v)\|=\ell(uv)$ for all $uv\in E$.
    It is a well-known fact that $G$ is generically (globally) $d$-rigid if and only
    if $\mathcal{S}_d$ is finite ($|\mathcal{S}_d|=1$);
    see, e.g., \cite{GortlerHealyThurston2007} for more details.
    The definitions of rigidity and global rigidity for frameworks with unit-length bars are therefore a natural analogue of generic rigidity and global rigidity: now, we simply substitute $\ell$ with a map that assigns 1 to each edge.
\end{remark}

In the all figures except for the last two we use the following colour code to explain different properties:
\begin{itemize}[noitemsep]
    \item Gray \iedge{colgraphe} and Blue \iedge{colB} describe $2$-trees, whose realisation is unique (see \Cref{lem:d_sphere_unique_realisation});
    \item Yellow \iedge{colO} and Green \iedge{colG} represent subgraphs that may flip;
    \item Red \iedge{colR} represents subgraphs that are rigid in some realisation but flexible in others;
    \item Purple \ivertex{colP} vertices are those that prevent some property.
\end{itemize}

\Cref{fig:pennyRigid} shows a penny graph which is \prigid\, but not globally \prigid.
There are the two penny realisations from the figure
since pinning the bottom grey part, the purple vertex is forced to lie on both dashed circles (the smaller circle indicates the restriction caused by removing the green block and the larger one describes the restriction caused by the green subgraph).
Further the green part can be reflected along the line between the two vertices where green and gray touch (dotted line),
yielding another two penny realisations.

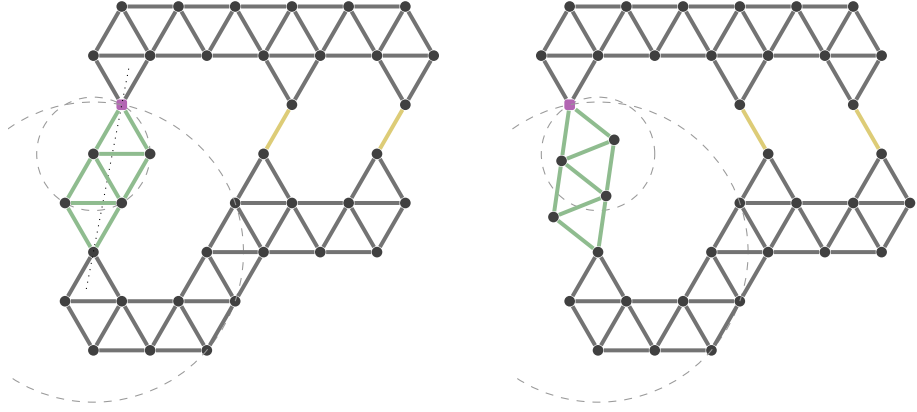
\begin{figure}[ht]
    \centering
    \begin{tikzpicture}[scale=0.75]
        \coordinate (dx) at (1,0);
        \coordinate (dy) at ($(0,{sqrt(3)/2})$);
        \coordinate (c) at ($0*(dx)+4*(dy)$);
        \clip ($(-1.5,-1)$) rectangle ($6*(dx)+7*(dy)+(0.2,0.2)$);
        \draw[dashed,black!40!white] (c) circle[radius=1cm];
        \node[vertex] (1) at ($0*(dx)+0*(dy)$) {};
        \node[vertex] (2) at ($1*(dx)+0*(dy)$) {};
        \node[vertex] (3) at ($2*(dx)+0*(dy)$) {};
        \node[vertex] (4) at ($-1/2*(dx)+1*(dy)$) {};
        \node[vertex] (5) at ($1/2*(dx)+1*(dy)$) {};
        \node[vertex] (6) at ($3/2*(dx)+1*(dy)$) {};
        \node[vertex] (7) at ($5/2*(dx)+1*(dy)$) {};
        \node[vertex] (8) at ($2*(dx)+2*(dy)$) {};
        \node[vertex] (9) at ($3*(dx)+2*(dy)$) {};
        \node[vertex] (10) at ($4*(dx)+2*(dy)$) {};
        \node[vertex] (11) at ($5*(dx)+2*(dy)$) {};
        \node[vertex] (15) at ($5/2*(dx)+3*(dy)$) {};
        \node[vertex] (16) at ($7/2*(dx)+3*(dy)$) {};
        \node[vertex] (17) at ($9/2*(dx)+3*(dy)$) {};
        \node[vertex] (18) at ($11/2*(dx)+3*(dy)$) {};
        \node[vertex] (33) at ($0*(dx)+2*(dy)$) {};
        \node[vertex] (34) at ($3*(dx)+4*(dy)$) {};
        \node[vertex] (35) at ($5*(dx)+4*(dy)$) {};
        \node[vertex] (38) at ($0*(dx)+6*(dy)$) {};
        \node[vertex] (39) at ($1*(dx)+6*(dy)$) {};
        \node[vertex] (40) at ($2*(dx)+6*(dy)$) {};
        \node[vertex] (41) at ($3*(dx)+6*(dy)$) {};
        \node[vertex] (42) at ($4*(dx)+6*(dy)$) {};
        \node[vertex] (43) at ($5*(dx)+6*(dy)$) {};
        \node[vertex] (44) at ($6*(dx)+6*(dy)$) {};
        \node[vertex] (45) at ($1/2*(dx)+7*(dy)$) {};
        \node[vertex] (46) at ($3/2*(dx)+7*(dy)$) {};
        \node[vertex] (47) at ($5/2*(dx)+7*(dy)$) {};
        \node[vertex] (48) at ($7/2*(dx)+7*(dy)$) {};
        \node[vertex] (49) at ($9/2*(dx)+7*(dy)$) {};
        \node[vertex] (50) at ($11/2*(dx)+7*(dy)$) {};
        \node[hvertex] (52) at ($1/2*(dx)+5*(dy)$) {};
        \node[vertex] (53) at ($7/2*(dx)+5*(dy)$) {};
        \node[vertex] (54) at ($11/2*(dx)+5*(dy)$) {};
        \node[vertex] (72) at ($-1/2*(dx)+3*(dy)$) {};
        \node[vertex] (73) at ($1/2*(dx)+3*(dy)$) {};
        \node[vertex] (74) at ($0*(dx)+4*(dy)$) {};
        \node[vertex] (75) at ($1*(dx)+4*(dy)$) {};
        \draw[dashed,black!40!white] (33) circle[radius={veclen(5/2,sqrt(3)/2)}];
        \draw[dotted] ($(33)!-0.25!(52)$)--($(33)!1.25!(52)$);
        \draw[edge] (1)--(2) (1)--(4) (1)--(5) (2)--(3) (2)--(5) (2)--(6) (3)--(6) (3)--(7) (4)--(5) (4)--(33) (5)--(6) (5)--(33) (6)--(7) (6)--(8) (7)--(8) (7)--(9) (8)--(9) (8)--(15) (9)--(10) (9)--(15) (9)--(16) (10)--(11) (10)--(16) (10)--(17) (11)--(17) (11)--(18) (15)--(16) (15)--(34) (16)--(17) (16)--(34) (17)--(18) (17)--(35) (18)--(35);
        \draw[edge,colO] (34)--(53) (35)--(54);
        \draw[edge] (38)--(39) (38)--(45) (38)--(52) (39)--(40) (39)--(45) (39)--(46) (39)--(52) (40)--(41) (40)--(46) (40)--(47) (41)--(42) (41)--(47) (41)--(48) (41)--(53) (42)--(43) (42)--(48) (42)--(49) (42)--(53) (43)--(44) (43)--(49) (43)--(50) (43)--(54) (44)--(50) (44)--(54) (45)--(46) (46)--(47) (47)--(48) (48)--(49) (49)--(50);
        \draw[edge,colG] (33)--(72) (33)--(73) (52)--(74) (52)--(75) (72)--(73) (72)--(74) (73)--(74) (73)--(75) (74)--(75);
    \end{tikzpicture}
    \qquad
    \begin{tikzpicture}[scale=0.75]
        \coordinate (dx) at (1,0);
        \coordinate (dy) at ($(0,{sqrt(3)/2})$);
        \coordinate (c) at ($0*(dx)+4*(dy)$);
        \clip ($(-1.5,-1)$) rectangle ($6*(dx)+7*(dy)+(0.2,0.2)$);
        \draw[dashed,black!40!white] (c) circle[radius=1cm];
        \node[vertex] (1) at ($0*(dx)+0*(dy)$) {};
        \node[vertex] (2) at ($1*(dx)+0*(dy)$) {};
        \node[vertex] (3) at ($2*(dx)+0*(dy)$) {};
        \node[vertex] (4) at ($-1/2*(dx)+1*(dy)$) {};
        \node[vertex] (5) at ($1/2*(dx)+1*(dy)$) {};
        \node[vertex] (6) at ($3/2*(dx)+1*(dy)$) {};
        \node[vertex] (7) at ($5/2*(dx)+1*(dy)$) {};
        \node[vertex] (8) at ($2*(dx)+2*(dy)$) {};
        \node[vertex] (9) at ($3*(dx)+2*(dy)$) {};
        \node[vertex] (10) at ($4*(dx)+2*(dy)$) {};
        \node[vertex] (11) at ($5*(dx)+2*(dy)$) {};
        \node[vertex] (15) at ($5/2*(dx)+3*(dy)$) {};
        \node[vertex] (16) at ($7/2*(dx)+3*(dy)$) {};
        \node[vertex] (17) at ($9/2*(dx)+3*(dy)$) {};
        \node[vertex] (18) at ($11/2*(dx)+3*(dy)$) {};
        \node[vertex] (33) at ($0*(dx)+2*(dy)$) {};
        \node[vertex] (34) at ($3*(dx)+4*(dy)$) {};
        \node[vertex] (35) at ($5*(dx)+4*(dy)$) {};
        \node[vertex] (38) at ($-1*(dx)+6*(dy)$) {};
        \node[vertex] (39) at ($0*(dx)+6*(dy)$) {};
        \node[vertex] (40) at ($1*(dx)+6*(dy)$) {};
        \node[vertex] (41) at ($2*(dx)+6*(dy)$) {};
        \node[vertex] (42) at ($3*(dx)+6*(dy)$) {};
        \node[vertex] (43) at ($4*(dx)+6*(dy)$) {};
        \node[vertex] (44) at ($5*(dx)+6*(dy)$) {};
        \node[vertex] (45) at ($-1/2*(dx)+7*(dy)$) {};
        \node[vertex] (46) at ($1/2*(dx)+7*(dy)$) {};
        \node[vertex] (47) at ($3/2*(dx)+7*(dy)$) {};
        \node[vertex] (48) at ($5/2*(dx)+7*(dy)$) {};
        \node[vertex] (49) at ($7/2*(dx)+7*(dy)$) {};
        \node[vertex] (50) at ($9/2*(dx)+7*(dy)$) {};
        \node[hvertex] (52) at ($-1/2*(dx)+5*(dy)$) {};
        \node[vertex] (53) at ($5/2*(dx)+5*(dy)$) {};
        \node[vertex] (54) at ($9/2*(dx)+5*(dy)$) {};
        \node[vertex] (72) at ($-11/14*(dx)+19/7*(dy)$) {};
        \node[vertex] (73) at ($1/7*(dx)+22/7*(dy)$) {};
        \node[vertex] (74) at ($-9/14*(dx)+27/7*(dy)$) {};
        \node[vertex] (75) at ($2/7*(dx)+30/7*(dy)$) {};
        \draw[dashed,black!40!white] (33) circle[radius={veclen(5/2,sqrt(3)/2)}];
        \draw[edge] (1)--(2) (1)--(4) (1)--(5) (2)--(3) (2)--(5) (2)--(6) (3)--(6) (3)--(7) (4)--(5) (4)--(33) (5)--(6) (5)--(33) (6)--(7) (6)--(8) (7)--(8) (7)--(9) (8)--(9) (8)--(15) (9)--(10) (9)--(15) (9)--(16) (10)--(11) (10)--(16) (10)--(17) (11)--(17) (11)--(18) (15)--(16) (15)--(34) (16)--(17) (16)--(34) (17)--(18) (17)--(35) (18)--(35);
        \draw[edge,colO] (34)--(53) (35)--(54);
        \draw[edge] (38)--(39) (38)--(45) (38)--(52) (39)--(40) (39)--(45) (39)--(46) (39)--(52) (40)--(41) (40)--(46) (40)--(47) (41)--(42) (41)--(47) (41)--(48) (41)--(53) (42)--(43) (42)--(48) (42)--(49) (42)--(53) (43)--(44) (43)--(49) (43)--(50) (43)--(54) (44)--(50) (44)--(54) (45)--(46) (46)--(47) (47)--(48) (48)--(49) (49)--(50);
        \draw[edge,colG] (33)--(72) (33)--(73) (52)--(74) (52)--(75) (72)--(73) (72)--(74) (73)--(74) (73)--(75) (74)--(75);
    \end{tikzpicture}
    \caption{A \prigid\ graph with two of four penny realisations, the other two are obtained by reflecting the green part along the dotted line.
    The dashed circles indicate the constraints for the purple vertex.
    }
    \label{fig:pennyRigid}
\end{figure}

Note that the underlying graph of a rigid $d$-sphere framework need not be $d$-sphere-rigid;
see \Cref{fig:genRigidPennyRigFlex},
which has a flexible and a rigid penny realisation.
Flipping the yellow subgraph as described in \Cref{fig:pennyRigid}
prevents the flex.
 
\begin{figure}[ht]
    \centering
    \begin{tikzpicture}[scale=0.75,rotate=90]
        \coordinate (dx) at (1,0);
        \coordinate (dy) at ($(0,{sqrt(3)/2})$);
        \coordinate (c) at ($0*(dx)+4*(dy)$);
        \node[vertex] (1) at ($0*(dx)+0*(dy)$) {};
        \node[vertex] (2) at ($1*(dx)+0*(dy)$) {};
        \node[vertex] (3) at ($2*(dx)+0*(dy)$) {};
        \node[vertex] (4) at ($-1/2*(dx)+1*(dy)$) {};
        \node[vertex] (5) at ($1/2*(dx)+1*(dy)$) {};
        \node[vertex] (6) at ($3/2*(dx)+1*(dy)$) {};
        \node[vertex] (7) at ($5/2*(dx)+1*(dy)$) {};
        \node[vertex] (8) at ($2*(dx)+2*(dy)$) {};
        \node[vertex] (9) at ($3*(dx)+2*(dy)$) {};
        \node[vertex] (10) at ($4*(dx)+2*(dy)$) {};
        \node[vertex] (11) at ($5*(dx)+2*(dy)$) {};
        \node[vertex] (12) at ($6*(dx)+2*(dy)$) {};
        \node[vertex] (13) at ($7*(dx)+2*(dy)$) {};
        \node[vertex] (14) at ($8*(dx)+2*(dy)$) {};
        \node[vertex] (15) at ($5/2*(dx)+3*(dy)$) {};
        \node[vertex] (16) at ($7/2*(dx)+3*(dy)$) {};
        \node[vertex] (17) at ($9/2*(dx)+3*(dy)$) {};
        \node[vertex] (18) at ($11/2*(dx)+3*(dy)$) {};
        \node[vertex] (19) at ($13/2*(dx)+3*(dy)$) {};
        \node[vertex] (20) at ($15/2*(dx)+3*(dy)$) {};
        \node[vertex] (21) at ($17/2*(dx)+3*(dy)$) {};
        \node[vertex] (22) at ($8*(dx)+4*(dy)$) {};
        \node[vertex] (23) at ($9*(dx)+4*(dy)$) {};
        \node[vertex] (24) at ($17/2*(dx)+5*(dy)$) {};
        \node[vertex] (25) at ($19/2*(dx)+5*(dy)$) {};
        \node[vertex] (26) at ($9*(dx)+6*(dy)$) {};
        \node[vertex] (27) at ($10*(dx)+6*(dy)$) {};
        \node[vertex] (28) at ($11*(dx)+6*(dy)$) {};
        \node[vertex] (29) at ($17/2*(dx)+7*(dy)$) {};
        \node[vertex] (30) at ($19/2*(dx)+7*(dy)$) {};
        \node[vertex] (31) at ($21/2*(dx)+7*(dy)$) {};
        \node[vertex] (32) at ($23/2*(dx)+7*(dy)$) {};
        \node[vertex] (33) at ($0*(dx)+2*(dy)$) {};
        \node[vertex] (34) at ($3*(dx)+4*(dy)$) {};
        \node[vertex] (35) at ($5*(dx)+4*(dy)$) {};
        \node[vertex] (36) at ($9*(dx)+8*(dy)$) {};
        \node[vertex] (37) at ($11*(dx)+8*(dy)$) {};
        \node[vertex] (38) at ($0*(dx)+6*(dy)$) {};
        \node[vertex] (39) at ($1*(dx)+6*(dy)$) {};
        \node[vertex] (40) at ($2*(dx)+6*(dy)$) {};
        \node[vertex] (41) at ($3*(dx)+6*(dy)$) {};
        \node[vertex] (42) at ($4*(dx)+6*(dy)$) {};
        \node[vertex] (43) at ($5*(dx)+6*(dy)$) {};
        \node[vertex] (44) at ($6*(dx)+6*(dy)$) {};
        \node[vertex] (45) at ($1/2*(dx)+7*(dy)$) {};
        \node[vertex] (46) at ($3/2*(dx)+7*(dy)$) {};
        \node[vertex] (47) at ($5/2*(dx)+7*(dy)$) {};
        \node[vertex] (48) at ($7/2*(dx)+7*(dy)$) {};
        \node[vertex] (49) at ($9/2*(dx)+7*(dy)$) {};
        \node[vertex] (50) at ($11/2*(dx)+7*(dy)$) {};
        \node[vertex] (51) at ($13/2*(dx)+7*(dy)$) {};
        \node[vertex] (52) at ($1/2*(dx)+5*(dy)$) {};
        \node[vertex] (53) at ($7/2*(dx)+5*(dy)$) {};
        \node[vertex] (54) at ($11/2*(dx)+5*(dy)$) {};
        \node[vertex] (55) at ($6*(dx)+8*(dy)$) {};
        \node[vertex] (56) at ($5*(dx)+10*(dy)$) {};
        \node[vertex] (57) at ($6*(dx)+10*(dy)$) {};
        \node[vertex] (58) at ($7*(dx)+10*(dy)$) {};
        \node[vertex] (59) at ($8*(dx)+10*(dy)$) {};
        \node[vertex] (60) at ($9*(dx)+10*(dy)$) {};
        \node[vertex] (61) at ($10*(dx)+10*(dy)$) {};
        \node[vertex] (62) at ($11*(dx)+10*(dy)$) {};
        \node[vertex] (63) at ($11/2*(dx)+11*(dy)$) {};
        \node[vertex] (64) at ($13/2*(dx)+11*(dy)$) {};
        \node[vertex] (65) at ($15/2*(dx)+11*(dy)$) {};
        \node[vertex] (66) at ($17/2*(dx)+11*(dy)$) {};
        \node[vertex] (67) at ($19/2*(dx)+11*(dy)$) {};
        \node[vertex] (68) at ($21/2*(dx)+11*(dy)$) {};
        \node[vertex] (69) at ($11/2*(dx)+9*(dy)$) {};
        \node[vertex] (70) at ($17/2*(dx)+9*(dy)$) {};
        \node[vertex] (71) at ($21/2*(dx)+9*(dy)$) {};
        \node[vertex] (72) at ($-1/2*(dx)+3*(dy)$) {};
        \node[vertex] (73) at ($1/2*(dx)+3*(dy)$) {};
        \node[vertex] (74) at ($0*(dx)+4*(dy)$) {};
        \node[vertex] (75) at ($1*(dx)+4*(dy)$) {};
        \draw[edge] (1)--(2) (1)--(4) (1)--(5) (2)--(3) (2)--(5) (2)--(6) (3)--(6) (3)--(7) (4)--(5) (4)--(33) (5)--(6) (5)--(33) (6)--(7) (6)--(8) (7)--(8) (7)--(9) (8)--(9) (8)--(15) (9)--(10) (9)--(15) (9)--(16) (10)--(11) (10)--(16) (10)--(17) (11)--(12) (11)--(17) (11)--(18) (12)--(13) (12)--(18) (12)--(19) (13)--(14) (13)--(19) (13)--(20) (14)--(20) (14)--(21) (15)--(16) (15)--(34) (16)--(17) (16)--(34) (17)--(18) (17)--(35) (18)--(19) (18)--(35) (19)--(20) (20)--(21) (20)--(22) (21)--(22) (21)--(23) (22)--(23) (22)--(24) (23)--(24) (23)--(25) (24)--(25) (24)--(26) (25)--(26) (25)--(27) (26)--(27) (26)--(29) (26)--(30) (27)--(28) (27)--(30) (27)--(31) (28)--(31) (28)--(32) (29)--(30) (29)--(36) (30)--(31) (30)--(36) (31)--(32) (31)--(37) (32)--(37); 
        \draw[edge,colO] (34)--(53) (35)--(54);
        \draw[edge] (38)--(39) (38)--(45) (38)--(52) (39)--(40) (39)--(45) (39)--(46) (39)--(52) (40)--(41) (40)--(46) (40)--(47) (41)--(42) (41)--(47) (41)--(48) (41)--(53) (42)--(43) (42)--(48) (42)--(49) (42)--(53) (43)--(44) (43)--(49) (43)--(50) (43)--(54) (44)--(50) (44)--(51) (44)--(54) (45)--(46) (46)--(47) (47)--(48) (48)--(49) (49)--(50) (50)--(51) (50)--(55) (51)--(55);
        \draw[edge,colR] (36)--(70) (37)--(71) (55)--(69);
        \draw[edge,colB] (56)--(57) (56)--(63) (56)--(69) (57)--(58) (57)--(63) (57)--(64) (57)--(69) (58)--(59) (58)--(64) (58)--(65) (59)--(60) (59)--(65) (59)--(66) (59)--(70) (60)--(61) (60)--(66) (60)--(67) (60)--(70) (61)--(62) (61)--(67) (61)--(68) (61)--(71) (62)--(68) (62)--(71) (63)--(64) (64)--(65) (65)--(66) (66)--(67) (67)--(68);
        \draw[edge,colO] (33)--(72) (33)--(73) (52)--(74) (52)--(75) (72)--(73) (72)--(74) (73)--(74) (73)--(75) (74)--(75);
    \end{tikzpicture}
    \quad
    \begin{tikzpicture}[scale=0.75,rotate=90]
        \coordinate (dx) at (1,0);
        \coordinate (dy) at ($(0,{sqrt(3)/2})$);
        \coordinate (c) at ($0*(dx)+4*(dy)$);
        \node[vertex] (1) at ($0*(dx)+0*(dy)$) {};
        \node[vertex] (2) at ($1*(dx)+0*(dy)$) {};
        \node[vertex] (3) at ($2*(dx)+0*(dy)$) {};
        \node[vertex] (4) at ($-1/2*(dx)+1*(dy)$) {};
        \node[vertex] (5) at ($1/2*(dx)+1*(dy)$) {};
        \node[vertex] (6) at ($3/2*(dx)+1*(dy)$) {};
        \node[vertex] (7) at ($5/2*(dx)+1*(dy)$) {};
        \node[vertex] (8) at ($2*(dx)+2*(dy)$) {};
        \node[vertex] (9) at ($3*(dx)+2*(dy)$) {};
        \node[vertex] (10) at ($4*(dx)+2*(dy)$) {};
        \node[vertex] (11) at ($5*(dx)+2*(dy)$) {};
        \node[vertex] (12) at ($6*(dx)+2*(dy)$) {};
        \node[vertex] (13) at ($7*(dx)+2*(dy)$) {};
        \node[vertex] (14) at ($8*(dx)+2*(dy)$) {};
        \node[vertex] (15) at ($5/2*(dx)+3*(dy)$) {};
        \node[vertex] (16) at ($7/2*(dx)+3*(dy)$) {};
        \node[vertex] (17) at ($9/2*(dx)+3*(dy)$) {};
        \node[vertex] (18) at ($11/2*(dx)+3*(dy)$) {};
        \node[vertex] (19) at ($13/2*(dx)+3*(dy)$) {};
        \node[vertex] (20) at ($15/2*(dx)+3*(dy)$) {};
        \node[vertex] (21) at ($17/2*(dx)+3*(dy)$) {};
        \node[vertex] (22) at ($8*(dx)+4*(dy)$) {};
        \node[vertex] (23) at ($9*(dx)+4*(dy)$) {};
        \node[vertex] (24) at ($17/2*(dx)+5*(dy)$) {};
        \node[vertex] (25) at ($19/2*(dx)+5*(dy)$) {};
        \node[vertex] (26) at ($9*(dx)+6*(dy)$) {};
        \node[vertex] (27) at ($10*(dx)+6*(dy)$) {};
        \node[vertex] (28) at ($11*(dx)+6*(dy)$) {};
        \node[vertex] (29) at ($17/2*(dx)+7*(dy)$) {};
        \node[vertex] (30) at ($19/2*(dx)+7*(dy)$) {};
        \node[vertex] (31) at ($21/2*(dx)+7*(dy)$) {};
        \node[vertex] (32) at ($23/2*(dx)+7*(dy)$) {};
        \node[vertex] (33) at ($0*(dx)+2*(dy)$) {};
        \node[vertex] (34) at ($3*(dx)+4*(dy)$) {};
        \node[vertex] (35) at ($5*(dx)+4*(dy)$) {};
        \node[vertex] (36) at ($9*(dx)+8*(dy)$) {};
        \node[vertex] (37) at ($11*(dx)+8*(dy)$) {};
        \node[vertex] (38) at ($-1*(dx)+6*(dy)$) {};
        \node[vertex] (39) at ($0*(dx)+6*(dy)$) {};
        \node[vertex] (40) at ($1*(dx)+6*(dy)$) {};
        \node[vertex] (41) at ($2*(dx)+6*(dy)$) {};
        \node[vertex] (42) at ($3*(dx)+6*(dy)$) {};
        \node[vertex] (43) at ($4*(dx)+6*(dy)$) {};
        \node[vertex] (44) at ($5*(dx)+6*(dy)$) {};
        \node[vertex] (45) at ($-1/2*(dx)+7*(dy)$) {};
        \node[vertex] (46) at ($1/2*(dx)+7*(dy)$) {};
        \node[vertex] (47) at ($3/2*(dx)+7*(dy)$) {};
        \node[vertex] (48) at ($5/2*(dx)+7*(dy)$) {};
        \node[vertex] (49) at ($7/2*(dx)+7*(dy)$) {};
        \node[vertex] (50) at ($9/2*(dx)+7*(dy)$) {};
        \node[vertex] (51) at ($11/2*(dx)+7*(dy)$) {};
        \node[vertex] (52) at ($-1/2*(dx)+5*(dy)$) {};
        \node[vertex] (53) at ($5/2*(dx)+5*(dy)$) {};
        \node[vertex] (54) at ($9/2*(dx)+5*(dy)$) {};
        \node[vertex] (55) at ($5*(dx)+8*(dy)$) {};
        \node[vertex] (56) at ($5*(dx)+10*(dy)$) {};
        \node[vertex] (57) at ($6*(dx)+10*(dy)$) {};
        \node[vertex] (58) at ($7*(dx)+10*(dy)$) {};
        \node[vertex] (59) at ($8*(dx)+10*(dy)$) {};
        \node[vertex] (60) at ($9*(dx)+10*(dy)$) {};
        \node[vertex] (61) at ($10*(dx)+10*(dy)$) {};
        \node[vertex] (62) at ($11*(dx)+10*(dy)$) {};
        \node[vertex] (63) at ($11/2*(dx)+11*(dy)$) {};
        \node[vertex] (64) at ($13/2*(dx)+11*(dy)$) {};
        \node[vertex] (65) at ($15/2*(dx)+11*(dy)$) {};
        \node[vertex] (66) at ($17/2*(dx)+11*(dy)$) {};
        \node[vertex] (67) at ($19/2*(dx)+11*(dy)$) {};
        \node[vertex] (68) at ($21/2*(dx)+11*(dy)$) {};
        \node[vertex] (69) at ($11/2*(dx)+9*(dy)$) {};
        \node[vertex] (70) at ($17/2*(dx)+9*(dy)$) {};
        \node[vertex] (71) at ($21/2*(dx)+9*(dy)$) {};
        \node[vertex] (72) at ($1/2*(dx)+3*(dy)$) {};
        \node[vertex] (73) at ($-1/2*(dx)+3*(dy)$) {};
        \node[vertex] (74) at ($0*(dx)+4*(dy)$) {};
        \node[vertex] (75) at ($-1*(dx)+4*(dy)$) {};
        \draw[edge] (1)--(2) (1)--(4) (1)--(5) (2)--(3) (2)--(5) (2)--(6) (3)--(6) (3)--(7) (4)--(5) (4)--(33) (5)--(6) (5)--(33) (6)--(7) (6)--(8) (7)--(8) (7)--(9) (8)--(9) (8)--(15) (9)--(10) (9)--(15) (9)--(16) (10)--(11) (10)--(16) (10)--(17) (11)--(12) (11)--(17) (11)--(18) (12)--(13) (12)--(18) (12)--(19) (13)--(14) (13)--(19) (13)--(20) (14)--(20) (14)--(21) (15)--(16) (15)--(34) (16)--(17) (16)--(34) (17)--(18) (17)--(35) (18)--(19) (18)--(35) (19)--(20) (20)--(21) (20)--(22) (21)--(22) (21)--(23) (22)--(23) (22)--(24) (23)--(24) (23)--(25) (24)--(25) (24)--(26) (25)--(26) (25)--(27) (26)--(27) (26)--(29) (26)--(30) (27)--(28) (27)--(30) (27)--(31) (28)--(31) (28)--(32) (29)--(30) (29)--(36) (30)--(31) (30)--(36) (31)--(32) (31)--(37) (32)--(37); 
        \draw[edge,colO] (34)--(53) (35)--(54);
        \draw[edge] (38)--(39) (38)--(45) (38)--(52) (39)--(40) (39)--(45) (39)--(46) (39)--(52) (40)--(41) (40)--(46) (40)--(47) (41)--(42) (41)--(47) (41)--(48) (41)--(53) (42)--(43) (42)--(48) (42)--(49) (42)--(53) (43)--(44) (43)--(49) (43)--(50) (43)--(54) (44)--(50) (44)--(51) (44)--(54) (45)--(46) (46)--(47) (47)--(48) (48)--(49) (49)--(50) (50)--(51) (50)--(55) (51)--(55);
        \draw[edge,colX] (36)--(70) (37)--(71) (55)--(69);
        \draw[edge,colB] (56)--(57) (56)--(63) (56)--(69) (57)--(58) (57)--(63) (57)--(64) (57)--(69) (58)--(59) (58)--(64) (58)--(65) (59)--(60) (59)--(65) (59)--(66) (59)--(70) (60)--(61) (60)--(66) (60)--(67) (60)--(70) (61)--(62) (61)--(67) (61)--(68) (61)--(71) (62)--(68) (62)--(71) (63)--(64) (64)--(65) (65)--(66) (66)--(67) (67)--(68);
        \draw[edge,colO] (33)--(72) (33)--(73) (52)--(74) (52)--(75) (72)--(73) (72)--(74) (73)--(74) (73)--(75) (74)--(75);
    \end{tikzpicture}
    \caption{Graph with a flexible (left) and rigid (right) penny framework. In the top framework, the blue subgraph can make a circular motion, while it cannot in the bottom one.}
    \label{fig:genRigidPennyRigFlex}
\end{figure}
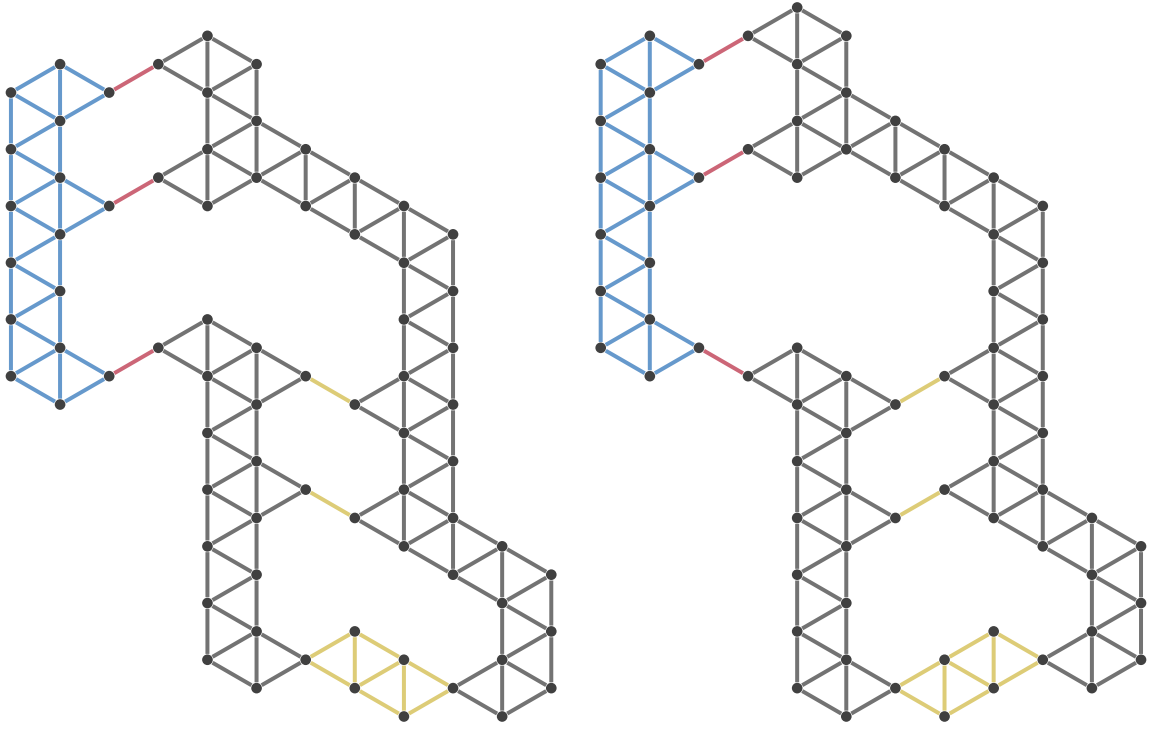

\section{Relations between unit-distance properties}
\label{sec:figs}
In this section, we compare unit-distance rigidity, sphere rigidity and matchstick rigidity. We start by the following simple observation that follows directly from the definitions.

\begin{remark}
    \label{rem:relationship}
    Every $d$-sphere graph is a $d$-unit-distance graph.
    Moreover, every penny graph is a matchstick graph and every matchstick graph is a $2$-unit-distance graph. The same is true for frameworks.
\end{remark}

In the previous examples (\Cref{fig:pennyRigid,fig:genRigidPennyRigFlex}), we implicitly use the fact that the 2-trees are globally \prigid, in particular the subgraphs indicated by gray edges in the figures.
Using an idea from the proof of \cite[Theorem~2.6]{DewarGraseggerEtAl2025},
we now give a formal proof of a slightly more general statement which,
in particular, guarantees that the $2$-tree subgraphs in the figures cannot have different realisations than shown.

\begin{lemma}
    \label{lem:d_sphere_unique_realisation}
    Let $G$ be a $d$-tree. If there are three vertices that have common neighbours inducing $K_d$,
    then $G$ is not a $d$-unit-distance graph.
    If $G$ is a $d$-sphere/matchstick/$d$-unit-distance graph,
    then $G$ is globally $d$-sphere/matchstick/$d$-unit-distance rigid;
    the matchstick case holds for $d=2$.
\end{lemma}
\begin{proof}
    Since there are only two unit-distance realisations of $K_{d+1}$ in $\mathbb{R}^d$
    up to direct isometries, no three distinct vertices can 
    have common neighbours inducing $K_d$ in a $d$-unit-distance realisation,
    as it is required to be injective.
    
    Hence, suppose that no three vertices have common neighbours inducing $K_d$
    and $G$ is a $d$-sphere/matchstick/$d$-unit-distance $d$-tree.
    We proceed by induction to prove the latter implication in the statement.
    The statement is clearly true for $K_{d+1}$.
    Now suppose the statement holds for all graphs on fewer vertices than $G$,
    and suppose that $G$ is a $d$-sphere/matchstick/$d$-unit-distance graph.
    Let $v$ be a degree $d$ vertex of $G$ such that $G'=G -v$ is a $d$-tree and $N_G(v)$ be the set of $d$ neighbours of $v$ in $G$.
    By the definition of $d$-tree, $N_G(v)$ is the vertex set of a complete graph $K_d$, which is contained in a $K_{d+1}$ subgraph of $G'$. 
    By our assumption, this implies that there is exactly one vertex $u\neq v$ such that $N_G(v)\subseteq N_G(u)$.
    Since $G$ is a $d$-sphere/matchstick/$d$-unit-distance graph, so is $G'$. Therefore,
    by the inductive hypothesis, there is a unique $d$-sphere/matchstick/$d$-unit-distance framework $(G',\rho')$ in $\mathbb{R}^d$ up to isometries.
    
    Since there are two unit-distance realisations of $K_{d+1}$ in~$\mathbb{R}^d$ up to direct isometries,
    there are exactly two points $A_1, A_2$ such that $\|A_i-\rho'(w)\|=1$ for all $w\in N_G(v)$.
    One of them is $\rho'(u)$, say $A_2$.
    Therefore, the realisation $\rho$ of $G$ in $\mathbb{R}^d$ defined by letting $\rho(v)=A_1$ and $\rho(w)=\rho'(w)$ for $w\in V(G')$ is the unique realisation of $G$ in $\mathbb{R}^d$ up to isometries.
\end{proof}

\begin{remark}
    Following the idea of the proof of \Cref{lem:d_sphere_unique_realisation},
    it can be checked in polynomially many operations over reals
    whether a $d$-tree is $d$-sphere/matchstick/$d$-unit-distance.
\end{remark}

\begin{remark}
    A $2$-tree is a matchstick graph if and only if it is a $2$-unit-distance graph.
    While one implication is \Cref{rem:relationship}, the other follows from the fact that
    a $2$-unit-distance realisation of a $2$-tree lies on a regular triangular grid, which is planar.
\end{remark}

Section 4 in \cite{DewarGraseggerEtAl2025} asks whether there is a \prigid\, (but not globally \prigid) graph that is generically globally rigid,
and analogously for marble-graphs.
It is easy to see that the graph given in \Cref{fig:genGlobPennyNonGlob} is \prigid\ by a similar argument as for the graph in \Cref{fig:pennyRigid} (and \Cref{lem:d_sphere_unique_realisation}).
Since the figure shows two distinct penny-realisations, this example answers the question in \cite{DewarGraseggerEtAl2025} positively.
The same idea can be extended to dimension three, as shown in \Cref{fig:genGlobMarbleNonGlob}, answering the analogous question for marble-graphs.
For both examples, the graph data and coordinates can be found in \cite{DewarGraseggerEtAl2026}. 
We have checked global rigidity computationally using \cite{AdrianGalletEtAl2026,Grasegger2022}.

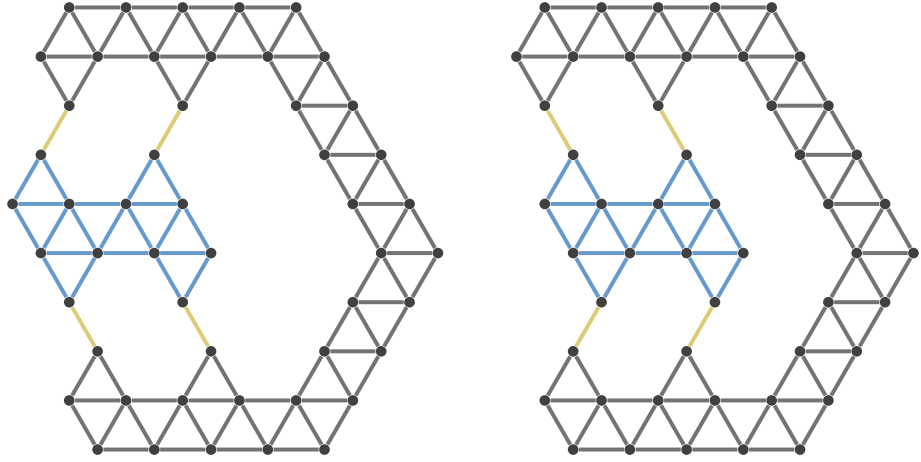
\begin{figure}[ht!]
    \centering
    \begin{tikzpicture}[scale=0.75]
        \coordinate (dx) at (1,0);
        \coordinate (dy) at ($(0,{sqrt(3)/2})$);
        \node[vertex] (1) at ($0*(dx)+0*(dy)$) {};
        \node[vertex] (2) at ($1*(dx)+0*(dy)$) {};
        \node[vertex] (3) at ($2*(dx)+0*(dy)$) {};
        \node[vertex] (4) at ($3*(dx)+0*(dy)$) {};
        \node[vertex] (5) at ($4*(dx)+0*(dy)$) {};
        \node[vertex] (6) at ($-1/2*(dx)+1*(dy)$) {};
        \node[vertex] (7) at ($1/2*(dx)+1*(dy)$) {};
        \node[vertex] (8) at ($3/2*(dx)+1*(dy)$) {};
        \node[vertex] (9) at ($5/2*(dx)+1*(dy)$) {};
        \node[vertex] (10) at ($7/2*(dx)+1*(dy)$) {};
        \node[vertex] (11) at ($9/2*(dx)+1*(dy)$) {};
        \node[vertex] (12) at ($4*(dx)+2*(dy)$) {};
        \node[vertex] (13) at ($5*(dx)+2*(dy)$) {};
        \node[vertex] (14) at ($9/2*(dx)+3*(dy)$) {};
        \node[vertex] (15) at ($11/2*(dx)+3*(dy)$) {};
        \node[vertex] (16) at ($5*(dx)+4*(dy)$) {};
        \node[vertex] (17) at ($6*(dx)+4*(dy)$) {};
        \node[vertex] (18) at ($9/2*(dx)+5*(dy)$) {};
        \node[vertex] (19) at ($11/2*(dx)+5*(dy)$) {};
        \node[vertex] (20) at ($4*(dx)+6*(dy)$) {};
        \node[vertex] (21) at ($5*(dx)+6*(dy)$) {};
        \node[vertex] (22) at ($7/2*(dx)+7*(dy)$) {};
        \node[vertex] (23) at ($9/2*(dx)+7*(dy)$) {};
        \node[vertex] (24) at ($-1*(dx)+8*(dy)$) {};
        \node[vertex] (25) at ($0*(dx)+8*(dy)$) {};
        \node[vertex] (26) at ($1*(dx)+8*(dy)$) {};
        \node[vertex] (27) at ($2*(dx)+8*(dy)$) {};
        \node[vertex] (28) at ($3*(dx)+8*(dy)$) {};
        \node[vertex] (29) at ($4*(dx)+8*(dy)$) {};
        \node[vertex] (30) at ($-1/2*(dx)+9*(dy)$) {};
        \node[vertex] (31) at ($1/2*(dx)+9*(dy)$) {};
        \node[vertex] (32) at ($3/2*(dx)+9*(dy)$) {};
        \node[vertex] (33) at ($5/2*(dx)+9*(dy)$) {};
        \node[vertex] (34) at ($7/2*(dx)+9*(dy)$) {};
        \node[vertex] (35) at ($0*(dx)+2*(dy)$) {};
        \node[vertex] (36) at ($2*(dx)+2*(dy)$) {};
        \node[vertex] (37) at ($-1/2*(dx)+7*(dy)$) {};
        \node[vertex] (38) at ($3/2*(dx)+7*(dy)$) {};
        \node[vertex] (39) at ($-1*(dx)+4*(dy)$) {};
        \node[vertex] (40) at ($0*(dx)+4*(dy)$) {};
        \node[vertex] (41) at ($1*(dx)+4*(dy)$) {};
        \node[vertex] (42) at ($2*(dx)+4*(dy)$) {};
        \node[vertex] (43) at ($-3/2*(dx)+5*(dy)$) {};
        \node[vertex] (44) at ($-1/2*(dx)+5*(dy)$) {};
        \node[vertex] (45) at ($1/2*(dx)+5*(dy)$) {};
        \node[vertex] (46) at ($3/2*(dx)+5*(dy)$) {};
        \node[vertex] (47) at ($-1/2*(dx)+3*(dy)$) {};
        \node[vertex] (48) at ($3/2*(dx)+3*(dy)$) {};
        \node[vertex] (49) at ($-1*(dx)+6*(dy)$) {};
        \node[vertex] (50) at ($1*(dx)+6*(dy)$) {};
        \draw[edge] (1)--(2) (1)--(6) (1)--(7) (2)--(3) (2)--(7) (2)--(8) (3)--(4) (3)--(8) (3)--(9) (4)--(5) (4)--(9) (4)--(10) (5)--(10) (5)--(11) (6)--(7) (6)--(35) (7)--(8) (7)--(35) (8)--(9) (8)--(36) (9)--(10) (9)--(36) (10)--(11) (10)--(12) (11)--(12) (11)--(13) (12)--(13) (12)--(14) (13)--(14) (13)--(15) (14)--(15) (14)--(16) (15)--(16) (15)--(17) (16)--(17) (16)--(18) (16)--(19) (17)--(19) (18)--(19) (18)--(20) (18)--(21) (19)--(21) (20)--(21) (20)--(22) (20)--(23) (21)--(23) (22)--(23) (22)--(28) (22)--(29) (23)--(29) (24)--(25) (24)--(30) (24)--(37) (25)--(26) (25)--(30) (25)--(31) (25)--(37) (26)--(27) (26)--(31) (26)--(32) (26)--(38) (27)--(28) (27)--(32) (27)--(33) (27)--(38) (28)--(29) (28)--(33) (28)--(34) (29)--(34) (30)--(31) (31)--(32) (32)--(33) (33)--(34);
        \draw[edge,colO] (35)--(47) (36)--(48) (37)--(49) (38)--(50);
        \draw[edge,colB] (39)--(40) (39)--(43) (39)--(44) (39)--(47) (40)--(41) (40)--(44) (40)--(45) (40)--(47) (41)--(42) (41)--(45) (41)--(46) (41)--(48) (42)--(46) (42)--(48) (43)--(44) (43)--(49) (44)--(45) (44)--(49) (45)--(46) (45)--(50) (46)--(50);
    \end{tikzpicture}
    \qquad
    \begin{tikzpicture}[scale=0.75]
        \coordinate (dx) at (1,0);
        \coordinate (dy) at ($(0,{sqrt(3)/2})$);
        \node[vertex] (1) at ($0*(dx)+0*(dy)$) {};
        \node[vertex] (2) at ($1*(dx)+0*(dy)$) {};
        \node[vertex] (3) at ($2*(dx)+0*(dy)$) {};
        \node[vertex] (4) at ($3*(dx)+0*(dy)$) {};
        \node[vertex] (5) at ($4*(dx)+0*(dy)$) {};
        \node[vertex] (6) at ($-1/2*(dx)+1*(dy)$) {};
        \node[vertex] (7) at ($1/2*(dx)+1*(dy)$) {};
        \node[vertex] (8) at ($3/2*(dx)+1*(dy)$) {};
        \node[vertex] (9) at ($5/2*(dx)+1*(dy)$) {};
        \node[vertex] (10) at ($7/2*(dx)+1*(dy)$) {};
        \node[vertex] (11) at ($9/2*(dx)+1*(dy)$) {};
        \node[vertex] (12) at ($4*(dx)+2*(dy)$) {};
        \node[vertex] (13) at ($5*(dx)+2*(dy)$) {};
        \node[vertex] (14) at ($9/2*(dx)+3*(dy)$) {};
        \node[vertex] (15) at ($11/2*(dx)+3*(dy)$) {};
        \node[vertex] (16) at ($5*(dx)+4*(dy)$) {};
        \node[vertex] (17) at ($6*(dx)+4*(dy)$) {};
        \node[vertex] (18) at ($9/2*(dx)+5*(dy)$) {};
        \node[vertex] (19) at ($11/2*(dx)+5*(dy)$) {};
        \node[vertex] (20) at ($4*(dx)+6*(dy)$) {};
        \node[vertex] (21) at ($5*(dx)+6*(dy)$) {};
        \node[vertex] (22) at ($7/2*(dx)+7*(dy)$) {};
        \node[vertex] (23) at ($9/2*(dx)+7*(dy)$) {};
        \node[vertex] (24) at ($-1*(dx)+8*(dy)$) {};
        \node[vertex] (25) at ($0*(dx)+8*(dy)$) {};
        \node[vertex] (26) at ($1*(dx)+8*(dy)$) {};
        \node[vertex] (27) at ($2*(dx)+8*(dy)$) {};
        \node[vertex] (28) at ($3*(dx)+8*(dy)$) {};
        \node[vertex] (29) at ($4*(dx)+8*(dy)$) {};
        \node[vertex] (30) at ($-1/2*(dx)+9*(dy)$) {};
        \node[vertex] (31) at ($1/2*(dx)+9*(dy)$) {};
        \node[vertex] (32) at ($3/2*(dx)+9*(dy)$) {};
        \node[vertex] (33) at ($5/2*(dx)+9*(dy)$) {};
        \node[vertex] (34) at ($7/2*(dx)+9*(dy)$) {};
        \node[vertex] (35) at ($0*(dx)+2*(dy)$) {};
        \node[vertex] (36) at ($2*(dx)+2*(dy)$) {};
        \node[vertex] (37) at ($-1/2*(dx)+7*(dy)$) {};
        \node[vertex] (38) at ($3/2*(dx)+7*(dy)$) {};
        \node[vertex] (39) at ($0*(dx)+4*(dy)$) {};
        \node[vertex] (40) at ($1*(dx)+4*(dy)$) {};
        \node[vertex] (41) at ($2*(dx)+4*(dy)$) {};
        \node[vertex] (42) at ($3*(dx)+4*(dy)$) {};
        \node[vertex] (43) at ($-1/2*(dx)+5*(dy)$) {};
        \node[vertex] (44) at ($1/2*(dx)+5*(dy)$) {};
        \node[vertex] (45) at ($3/2*(dx)+5*(dy)$) {};
        \node[vertex] (46) at ($5/2*(dx)+5*(dy)$) {};
        \node[vertex] (47) at ($1/2*(dx)+3*(dy)$) {};
        \node[vertex] (48) at ($5/2*(dx)+3*(dy)$) {};
        \node[vertex] (49) at ($0*(dx)+6*(dy)$) {};
        \node[vertex] (50) at ($2*(dx)+6*(dy)$) {};
        \draw[edge] (1)--(2) (1)--(6) (1)--(7) (2)--(3) (2)--(7) (2)--(8) (3)--(4) (3)--(8) (3)--(9) (4)--(5) (4)--(9) (4)--(10) (5)--(10) (5)--(11) (6)--(7) (6)--(35) (7)--(8) (7)--(35) (8)--(9) (8)--(36) (9)--(10) (9)--(36) (10)--(11) (10)--(12) (11)--(12) (11)--(13) (12)--(13) (12)--(14) (13)--(14) (13)--(15) (14)--(15) (14)--(16) (15)--(16) (15)--(17) (16)--(17) (16)--(18) (16)--(19) (17)--(19) (18)--(19) (18)--(20) (18)--(21) (19)--(21) (20)--(21) (20)--(22) (20)--(23) (21)--(23) (22)--(23) (22)--(28) (22)--(29) (23)--(29) (24)--(25) (24)--(30) (24)--(37) (25)--(26) (25)--(30) (25)--(31) (25)--(37) (26)--(27) (26)--(31) (26)--(32) (26)--(38) (27)--(28) (27)--(32) (27)--(33) (27)--(38) (28)--(29) (28)--(33) (28)--(34) (29)--(34) (30)--(31) (31)--(32) (32)--(33) (33)--(34);
        \draw[edge,colO] (35)--(47) (36)--(48) (37)--(49) (38)--(50);
        \draw[edge,colB] (39)--(40) (39)--(43) (39)--(44) (39)--(47) (40)--(41) (40)--(44) (40)--(45) (40)--(47) (41)--(42) (41)--(45) (41)--(46) (41)--(48) (42)--(46) (42)--(48) (43)--(44) (43)--(49) (44)--(45) (44)--(49) (45)--(46) (45)--(50) (46)--(50);
    \end{tikzpicture}
    \caption{A globally 2-rigid graph with two penny realisations.}
    \label{fig:genGlobPennyNonGlob}
\end{figure}

For the remainder of the section we restrict to $d=2$. We now establish further connections between penny and matchstick rigidity properties. See \Cref{tab:comparisontab} (yellow) for a summary.

\begin{table}[ht]
    \centering
    \begin{tikzpicture}[
        head/.style={align=center,minimum width=1.7cm,minimum height=1.5cm},
        mhead/.style={head,anchor=east,align=right,minimum width=1.5cm,minimum height=1.3cm},
        phead/.style={head,anchor=south west,},
        uhead/.style={head,anchor=north west},
        entry/.style={minimum width=1.7cm,minimum height=1.3cm,align=center,rectangle,anchor=west},
        cap/.style={},
        mcap/.style={cap},
        ucap/.style={cap},
        pcap/.style={cap},
        fillback/.style={opacity=0.5,rounded corners=6pt},
        every node/.style={font=\scriptsize}
    ]
        \coordinate (mleft) at (-0.1,0);
        \coordinate (h) at (0,-1.3);
        \node[mhead] (m1) at (mleft) {not\\\mrigid};
        \node[mhead] (m2) at ($(mleft)+(h)$) {\mrigid\\not glob.\\\mrigid};
        \node[mhead] (m3) at ($(mleft)+2*(h)$) {globally\\\mrigid};
        \node[rotate=90] at ($(m2.west)-(0.3,0)$) {matchstick rigidity};

        \coordinate (pup) at (0,0.75);
        \coordinate (w) at (1.7,0);
        \node[phead] (p1) at (pup) {not\\\prigid};
        \node[phead] (p2) at ($(pup)+(w)$) {\prigid\\not glob.\\\prigid};
        \node[phead] (p3) at ($(pup)+2*(w)$) {globally\\\prigid};
        \node[] at ($(p2.north)+(0,0.3)$) {penny rigidity};

        \coordinate (udown) at (0,-0.75);
        \node[uhead] (u1) at ($(udown)+2*(h)$) {not\\\urigid};
        \node[uhead] (u2) at ($(udown)+2*(h)+(w)$) {\urigid\\not glob.\\\urigid};
        \node[uhead] (u3) at ($(udown)+2*(h)+2*(w)$) {globally\\\urigid};
        \node[] at ($(u2.south)+(0,-0.3)$) {unit-distance rigidity};

        \node[entry] (e00) at (0,0) {4-cycle};
        \node[entry] (e11) at ($(w)+(h)$) {\Cref{fig:pennyMatchRigid}\\(left)};
        \node[entry] (e22) at ($2*(w)+2*(h)$) {
            \begin{tikzpicture}[scale=0.4]
                \node[vertex] (a1) at (0,0) {};
                \node[vertex] (a2) at (0:1) {};
                \node[vertex] (a3) at (0:2) {};
                \node[vertex] (b1) at (60:1) {};
                \path (60:1) ++(1,0) node[vertex] (b2) {};
                \node[vertex] (c1) at (60:2) {};
                \draw[edge] (a1)--(a2) (a2)--(a3) (a1)--(b1) (a2)--(b1) (a2)--(b2) (a3)--(b2) (b1)--(b2) (b1)--(c1) (b2)--(c1);
            \end{tikzpicture}
        };

        \node[entry] (e01) at ($(h)$) {\Cref{fig:matchRigUnitFlex}\\(without\\purple)};
        \node[entry] (e02) at ($2*(h)$) {\Cref{fig:matchRigUnitFlex}};
        \node[entry] (e12) at ($(w)+2*(h)$) {\Cref{fig:matchGlobRigUnitRig}};

        \node[entry] (e10) at ($(w)$) {\Cref{fig:pennyRigidMatchFlex}};
        \node[entry] (e20) at ($2*(w)$) {\Cref{fig:pennyGlobRigMatchFlex}};
        \node[entry] (e21) at ($2*(w)+(h)$) {\Cref{fig:pennyMatchRigid}\\(right)};

        \begin{pgfonlayer}{background}
            \fill[fillback,black!5!white] (-2.5,3)rectangle(5.75,-5.5);
            \fill[fillback,colB] (e02.south west)--(e22.south east)--(e22.north east)--(e22.north west)--(e11.north east)--(e11.north west)--(e00.north east)--(e00.north west)--cycle;
            \fill[fillback,colO] (e00.north west)--(e20.north east)--(e22.south east)--(e22.south west)--(e22.north west)--(e11.south west)--(e11.north west)--(e00.south west)--cycle;

            \fill[fillback,black!25!white] (m1.north west)--(m3.south west)--(m3.south east)--(m1.north east)--cycle;
            \fill[fillback,colB,opacity=0.7] (u1.north west)--(u1.south west)--(u3.south east)--(u3.north east)--cycle;
            \fill[fillback,colO,opacity=0.7] (p1.north west)--(p1.south west)--(p3.south east)--(p3.north east)--cycle;
        \end{pgfonlayer}        
    \end{tikzpicture}    
    \caption{Comparison of penny, 2-unit-distance and matchstick rigidity properties.
    The blue part refers to unit-distance rigidity, the yellow one to penny rigidity, and the diagonal covers both.
    Cases that are not covered do not exist by \Cref{rem:relationship}.}
    \label{tab:comparisontab}
\end{table}

\begin{figure}[!h]
    \centering
    \includegraphics[trim={0 0.4cm 0 0.3cm},clip,width=0.45\textwidth]{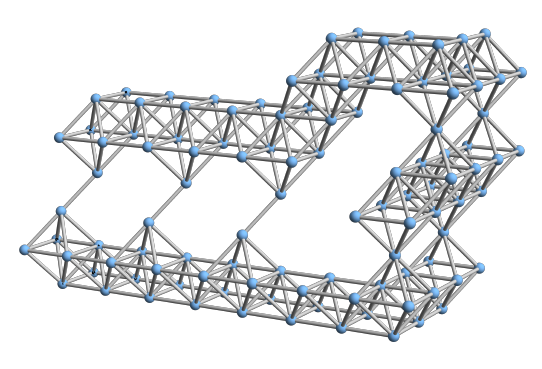}
    \includegraphics[trim={0 0.4cm 0 0.3cm},width=0.45\textwidth]{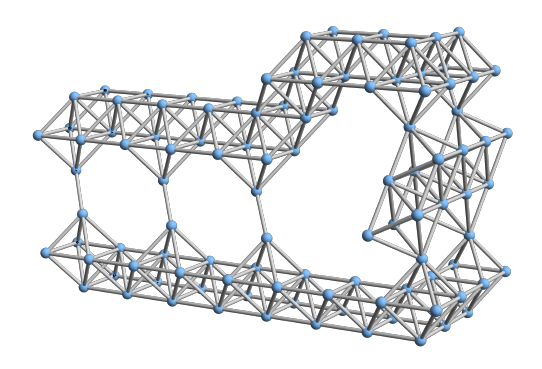}
    \caption{A globally 3-rigid graph with two of its finitely many marble realisations.}
    \label{fig:genGlobMarbleNonGlob}
\end{figure}

We first note that there is a \prigid\, graph which is neither \mrigid\, nor globally \prigid. The graph in \Cref{fig:pennyRigidMatchFlex} is not \mrigid, since the red path in the realisation shown on the left is flexible.
There are exactly two penny realisations of the graph: one shown on the right and one obtained from this by reflecting the green subgraph as described for \Cref{fig:pennyRigid}. 

\begin{figure}[ht]
    \centering
    \begin{tikzpicture}[scale=0.75,rotate=90]
        \coordinate (dx) at (0:1);
        \coordinate (dy) at (60:1);
        \coordinate (dyy) at (120:1);
        \node[vertex] (bt-1) at (dyy) {};
        \foreach \i [remember=\i as \io (initially -1)] in {0,1,2}
        {
            \node[vertex] (b\i) at ($\i*(dx)$) {};
            \node[vertex] (bt\i) at ($\i*(dx)+(dy)$) {};
            \draw[edge] (b\i)--(bt\i) (b\i)--(bt\io) (bt\i)--(bt\io);
            \ifnum\i>0
                \draw[edge] (b\i)--(b\io);
            \fi
        }
        \foreach \i [remember=\i as \io (initially 0)] in {3,4,...,8}
        {
            \node[vertex] (b\i) at ($-2*(dx)+\i*(dx)+2*(dy)$) {};
            \node[vertex] (bt\i) at ($-2*(dx)+\i*(dx)+3*(dy)$) {};
            \draw[edge] (b\i)--(bt\i);
            \ifnum\i>3
                \draw[edge] (bt\i)--(bt\io) (b\i)--(bt\io) (b\i)--(b\io);
            \fi
        }
        \draw[edge] (bt1)--(b3) (bt2)--(b3) (bt2)--(b4);
        \foreach \k [count=\i,evaluate=\k as \kp using int(\k+1)] in {-1,3,5}
        {
            \node[vertex] (s\i) at ($(bt\k)+(dy)$) {};
            \draw[edge] (s\i)--(bt\k) (s\i)--(bt\kp);
        }
        \foreach \i [remember=\i as \io (initially 0)] in {1,2,3}
        {
            \node[vertex] (cl\i) at ($(bt7)+\i*(dy)$) {};
            \node[vertex] (cr\i) at ($(bt8)+\i*(dy)$) {};
            \draw[edge] (cl\i)--(cr\i);
            \ifnum\i>1
                \draw[edge] (cl\i)--(cl\io) (cr\i)--(cr\io) (cl\i)--(cr\io);
            \fi
        }
        \node[vertex] (cl4) at ($(bt7)+4*(dy)$) {};
        \draw[edge] (cl4)--(cl3) (cl4)--(cr3);
        \draw[edge] (cl1)--(bt7) (cl1)--(bt8) (cr1)--(bt8);

        \foreach \i [remember=\i as \io (initially 0)] in {1,2,...,7}
        {
            \node[vertex] (d\i) at ($-4*(dx)+\i*(dx)+6*(dy)$) {};
            \node[vertex] (dt\i) at ($-4*(dx)+\i*(dx)+7*(dy)$) {};
            \draw[edge] (d\i)--(dt\i);
            \ifnum\i>1
                \draw[edge] (dt\i)--(dt\io) (d\i)--(dt\io) (d\i)--(d\io);
            \fi
        }
        \foreach \k [count=\i,evaluate=\k as \kp using int(\k+1)] in {1,4,6}
        {
            \node[vertex] (t\i) at ($(d\k)-(dyy)$) {};
            \draw[edge] (t\i)--(d\k) (t\i)--(d\kp);
        }
        \node[hvertex] (p) at ($(cl1)+(dyy)$) {};
        \draw[edge] (p)--(cl1) (p)--(cl2);

        \node[vertex] (f1) at ($(s1)+(dyy)$) {};
        \node[vertex] (f2) at ($(s1)+(dy)$) {};
        \draw[edge,colG] (f1)--(f2) (f1)--(s1) (f2)--(s1);
        \node[vertex] (f3) at ($(f1)+(dy)$) {};
        \node[vertex] (f4) at ($(f2)+(dy)$) {};
        \draw[edge,colG] (f3)--(f4) (f3)--(f1) (f3)--(f2) (f3)--(t1) (f4)--(f2) (f4)--(t1);
        \draw[edge,colO] (s2)--(t2) (s3)--(t3);

        \node[vertex] (e1) at ($(dt7)+(40:1)$) {};
        \path[name path=C1] (e1) circle[radius=1cm];
        \path[name path=C2] (cl4) circle[radius=1cm];
        \path[name intersections={of=C1 and C2}];
        \node[vertex] (e2) at (intersection-1) {};
        \draw[edge,colR] (dt7)--(e1) (e1)--(e2) (e2)--(cl4);
        \node[hvertex] at (d7) {};
    \end{tikzpicture}
    \qquad
    \begin{tikzpicture}[scale=0.75,rotate=90]
        \coordinate (dx) at (0:1);
        \coordinate (dy) at (60:1);
        \coordinate (dyy) at (120:1);
        \node[vertex] (bt-1) at (dyy) {};
        \foreach \i [remember=\i as \io (initially -1)] in {0,1,2}
        {
            \node[vertex] (b\i) at ($\i*(dx)$) {};
            \node[vertex] (bt\i) at ($\i*(dx)+(dy)$) {};
            \draw[edge] (b\i)--(bt\i) (b\i)--(bt\io) (bt\i)--(bt\io);
            \ifnum\i>0
                \draw[edge] (b\i)--(b\io);
            \fi
        }
        \foreach \i [remember=\i as \io (initially 0)] in {3,4,...,8}
        {
            \node[vertex] (b\i) at ($-2*(dx)+\i*(dx)+2*(dy)$) {};
            \node[vertex] (bt\i) at ($-2*(dx)+\i*(dx)+3*(dy)$) {};
            \draw[edge] (b\i)--(bt\i);
            \ifnum\i>3
                \draw[edge] (bt\i)--(bt\io) (b\i)--(bt\io) (b\i)--(b\io);
            \fi
        }
        \draw[edge] (bt1)--(b3) (bt2)--(b3) (bt2)--(b4);
        \foreach \k [count=\i,evaluate=\k as \kp using int(\k+1)] in {-1,3,5}
        {
            \node[vertex] (s\i) at ($(bt\k)+(dy)$) {};
            \draw[edge] (s\i)--(bt\k) (s\i)--(bt\kp);
        }
        \foreach \i [remember=\i as \io (initially 0)] in {1,2,3}
        {
            \node[vertex] (cl\i) at ($(bt7)+\i*(dy)$) {};
            \node[vertex] (cr\i) at ($(bt8)+\i*(dy)$) {};
            \draw[edge] (cl\i)--(cr\i);
            \ifnum\i>1
                \draw[edge] (cl\i)--(cl\io) (cr\i)--(cr\io) (cl\i)--(cr\io);
            \fi
        }
        \node[vertex] (cl4) at ($(bt7)+4*(dy)$) {};
        \draw[edge] (cl4)--(cl3) (cl4)--(cr3);
        \draw[edge] (cl1)--(bt7) (cl1)--(bt8) (cr1)--(bt8);

        \foreach \i [remember=\i as \io (initially 0)] in {1,2,...,7}
        {
            \node[vertex] (d\i) at ($-5*(dx)+\i*(dx)+6*(dy)$) {};
            \node[vertex] (dt\i) at ($-5*(dx)+\i*(dx)+7*(dy)$) {};
            \draw[edge] (d\i)--(dt\i);
            \ifnum\i>1
                \draw[edge] (dt\i)--(dt\io) (d\i)--(dt\io) (d\i)--(d\io);
            \fi
        }
        \foreach \k [count=\i,evaluate=\k as \kp using int(\k+1)] in {1,4,6}
        {
            \node[vertex] (t\i) at ($(d\k)-(dyy)$) {};
            \draw[edge] (t\i)--(d\k) (t\i)--(d\kp);
        }
        \node[vertex] (p) at ($(cl1)+(dyy)$) {};
        \draw[edge] (p)--(cl1) (p)--(cl2);

        \node[vertex] (f1) at ($(s1)+(dyy)$) {};
        \node[vertex] (f2) at ($(s1)+(dy)$) {};
        \draw[edge,colG] (f1)--(f2) (f1)--(s1) (f2)--(s1);
        \node[vertex] (f3) at ($(f1)+(dyy)$) {};
        \node[vertex] (f4) at ($(f2)+(dyy)$) {};
        \draw[edge,colG] (f3)--(f4) (f3)--(f1) (f4)--(f1) (f3)--(t1) (f4)--(f2) (f4)--(t1);
        \draw[edge,colO] (s2)--(t2) (s3)--(t3);

        \node[vertex] (e1) at ($(dt7)+(dx)$) {};
        \node[vertex] (e2) at ($(e1)+(dx)$) {};
        \draw[edge,colX] (dt7)--(e1) (e1)--(e2) (e2)--(cl4);
    \end{tikzpicture}
    \caption{A graph with a flexible matchstick realisation (left) that is not a penny realisation since the two purple vertices are too close and a rigid penny realisation (right).}
    \label{fig:pennyRigidMatchFlex}
\end{figure}
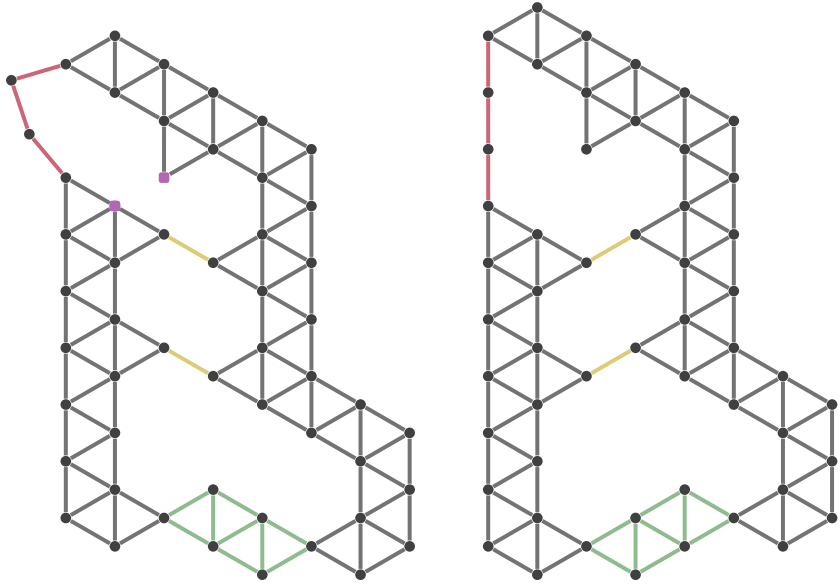

There are also graphs which are globally \prigid\, and not \mrigid. For instance, the graph in \Cref{fig:pennyGlobRigMatchFlex} has a unique penny realisation but a flexible realisation that has matchstick but not penny properties.

\begin{figure}[ht]
    \centering
    \begin{tikzpicture}[scale=0.75,rotate=-90]
        \coordinate (dx) at (0:1);
        \coordinate (dy) at (60:1);
        \coordinate (dyy) at (120:1);
        \node[vertex] (bt-1) at (dyy) {};
        \foreach \i [remember=\i as \io (initially -1)] in {0,1,...,4}
        {
            \node[vertex] (b\i) at ($\i*(dx)$) {};
            \node[vertex] (bt\i) at ($\i*(dx)+(dy)$) {};
            \draw[edge] (b\i)--(bt\i) (b\i)--(bt\io) (bt\i)--(bt\io);
            \ifnum\i>0
                \draw[edge] (b\i)--(b\io);
            \fi
        }
        \foreach \i [remember=\i as \io (initially 0)] in {1,...,3}
        {
            \node[vertex] (rr\i) at ($(bt4)+\i*(dy)$) {};
            \node[vertex] (rl\i) at ($(bt3)+\i*(dy)$) {};
            \draw[edge] (rr\i)--(rl\i);
            \ifnum\i>1
                \draw[edge] (rr\i)--(rr\io) (rl\i)--(rl\io) (rl\i)--(rr\io);
            \fi
        }
        \draw[edge] (rr1)--(bt4) (rl1)--(bt3) (rl1)--(bt4);
        \foreach \i [count=\k, remember=\i as \io (initially 3)] in {4,...,8}
        {
            \node[vertex] (rr\i) at ($(rr3)+\k*(dyy)$) {};
            \node[vertex] (rl\i) at ($(rl3)+\k*(dyy)$) {};
            \draw[edge] (rr\i)--(rl\i) (rr\i)--(rr\io) (rl\i)--(rl\io) (rr\i)--(rl\io);
        }
        \foreach \i [remember=\i as \io (initially 0)] in {1,...,3}
        {
            \node[vertex] (c\i) at ($(rl7)-\i*(dx)$) {};
            \node[vertex] (ct\i) at ($(rl8)-\i*(dx)$) {};
            \draw[edge] (c\i)--(ct\i);
            \ifnum\i>1
                \draw[edge] (c\i)--(c\io) (c\i)--(ct\io) (ct\i)--(ct\io);
            \fi
        }
        \draw[edge] (c1)--(rl7) (ct1)--(rl8) (c1)--(rl8);
        \node[vertex] (c4) at ($(rl7)-4*(dx)$) {};
        \draw[edge] (c4)--(c3) (c4)--(ct3);
        \node[vertex] (x) at ($(rl3)-(dx)$) {};
        \draw[edge] (x)--(rl2) (x)--(rl3) (x)--(rl4);
        \foreach \i [remember=\i as \io (initially 0)] in {1,...,4}
        {
            \node[vertex] (m\i) at ($(b1)+4*(dyy)+\i*(dx)$) {};
            \node[vertex] (mt\i) at ($(b1)+5*(dyy)+\i*(dx)$) {};
            \draw[edge] (m\i)--(mt\i);
            \ifnum\i>1
                \draw[edge] (m\i)--(m\io) (mt\i)--(mt\io) (mt\i)--(m\io);
            \fi
        }
        \foreach \p/\q [count=\i] in {bt-1/bt0,bt1/bt2,mt1/mt2,mt3/mt4}
        {
            \node[vertex] (s\i) at ($(\p)+(dy)$) {};
            \draw[edge] (s\i)--(\p) (s\i)--(\q);
        }
        \foreach \p/\q [count=\i] in {m1/m2,m3/m4,c4/c3,c2/c1}
        {
            \node[vertex] (t\i) at ($(\p)-(dyy)$) {};
            \draw[edge] (t\i)--(\p) (t\i)--(\q);
            \draw[edge,colO] (s\i)--(t\i);
        }

        \foreach \i [remember=\i as \io (initially 0)] in {1,...,3}
        {
            \node[vertex] (n\i) at ($(t1)-\i*(dx)$) {};
            \node[vertex] (nt\i) at ($(m1)-\i*(dx)$) {};
            \draw[edge] (n\i)--(nt\i);
            \ifnum\i>1
                \draw[edge] (n\i)--(n\io) (nt\i)--(nt\io) (n\i)--(nt\io);
            \fi
        }
        \draw[edge] (n1)--(t1) (nt1)--(m1) (n1)--(m1)  (nt1)--(mt1);
        \foreach \i [count=\k,remember=\i as \io (initially 3)] in {4,5}
        {
            \node[vertex] (n\i) at ($(n3)-\k*(dy)$) {};
            \node[vertex] (nt\i) at ($(nt3)-\k*(dy)$) {};
            \draw[edge] (n\i)--(nt\i);
            \ifnum\i>1
                \draw[edge] (n\i)--(n\io) (nt\i)--(nt\io) (nt\i)--(n\io);
            \fi
        }
        \node[vertex] (n6) at ($(n4)-(dyy)$) {};
        \draw[edge] (n6)--(n4) (n6)--(n5);
        \node[vertex] (e1) at ($(n6)+(-40:1)$) {};
        \path[name path=C1] (e1) circle[radius=1cm];
        \path[name path=C2] (bt-1) circle[radius=1cm];
        \path[name intersections={of=C1 and C2}];
        \node[vertex] (e2) at (intersection-1) {};
        \draw[edge,colR] (n6)--(e1) (e1)--(e2) (e2)--(bt-1);
        \node[hvertex] at (x) {};
        \node[hvertex] at (m4) {};
    \end{tikzpicture}
    \qquad
    \begin{tikzpicture}[scale=0.75,rotate=-90]
        \coordinate (dx) at (0:1);
        \coordinate (dy) at (60:1);
        \coordinate (dyy) at (120:1);
        \node[vertex] (bt-1) at (dyy) {};
        \foreach \i [remember=\i as \io (initially -1)] in {0,1,...,4}
        {
            \node[vertex] (b\i) at ($\i*(dx)$) {};
            \node[vertex] (bt\i) at ($\i*(dx)+(dy)$) {};
            \draw[edge] (b\i)--(bt\i) (b\i)--(bt\io) (bt\i)--(bt\io);
            \ifnum\i>0
                \draw[edge] (b\i)--(b\io);
            \fi
        }
        \foreach \i [remember=\i as \io (initially 0)] in {1,...,3}
        {
            \node[vertex] (rr\i) at ($(bt4)+\i*(dy)$) {};
            \node[vertex] (rl\i) at ($(bt3)+\i*(dy)$) {};
            \draw[edge] (rr\i)--(rl\i);
            \ifnum\i>1
                \draw[edge] (rr\i)--(rr\io) (rl\i)--(rl\io) (rl\i)--(rr\io);
            \fi
        }
        \draw[edge] (rr1)--(bt4) (rl1)--(bt3) (rl1)--(bt4);
        \foreach \i [count=\k, remember=\i as \io (initially 3)] in {4,...,8}
        {
            \node[vertex] (rr\i) at ($(rr3)+\k*(dyy)$) {};
            \node[vertex] (rl\i) at ($(rl3)+\k*(dyy)$) {};
            \draw[edge] (rr\i)--(rl\i) (rr\i)--(rr\io) (rl\i)--(rl\io) (rr\i)--(rl\io);
        }
        \foreach \i [remember=\i as \io (initially 0)] in {1,...,3}
        {
            \node[vertex] (c\i) at ($(rl7)-\i*(dx)$) {};
            \node[vertex] (ct\i) at ($(rl8)-\i*(dx)$) {};
            \draw[edge] (c\i)--(ct\i);
            \ifnum\i>1
                \draw[edge] (c\i)--(c\io) (c\i)--(ct\io) (ct\i)--(ct\io);
            \fi
        }
        \draw[edge] (c1)--(rl7) (ct1)--(rl8) (c1)--(rl8);
        \node[vertex] (c4) at ($(rl7)-4*(dx)$) {};
        \draw[edge] (c4)--(c3) (c4)--(ct3);
        \node[vertex] (x) at ($(rl3)-(dx)$) {};
        \draw[edge] (x)--(rl2) (x)--(rl3) (x)--(rl4);
        \foreach \i [remember=\i as \io (initially 0)] in {1,...,4}
        {
            \node[vertex] (m\i) at ($(b1)+4*(dyy)-(dx)+\i*(dx)$) {};
            \node[vertex] (mt\i) at ($(b1)+5*(dyy)-(dx)+\i*(dx)$) {};
            \draw[edge] (m\i)--(mt\i);
            \ifnum\i>1
                \draw[edge] (m\i)--(m\io) (mt\i)--(mt\io) (mt\i)--(m\io);
            \fi
        }
        \foreach \p/\q [count=\i] in {bt-1/bt0,bt1/bt2,mt1/mt2,mt3/mt4}
        {
            \node[vertex] (s\i) at ($(\p)+(dy)$) {};
            \draw[edge] (s\i)--(\p) (s\i)--(\q);
        }
        \foreach \p/\q [count=\i] in {m1/m2,m3/m4,c4/c3,c2/c1}
        {
            \node[vertex] (t\i) at ($(\p)-(dyy)$) {};
            \draw[edge] (t\i)--(\p) (t\i)--(\q);
            \draw[edge,colO] (s\i)--(t\i);
        }

        \foreach \i [remember=\i as \io (initially 0)] in {1,...,3}
        {
            \node[vertex] (n\i) at ($(t1)-\i*(dx)$) {};
            \node[vertex] (nt\i) at ($(m1)-\i*(dx)$) {};
            \draw[edge] (n\i)--(nt\i);
            \ifnum\i>1
                \draw[edge] (n\i)--(n\io) (nt\i)--(nt\io) (n\i)--(nt\io);
            \fi
        }
        \draw[edge] (n1)--(t1) (nt1)--(m1) (n1)--(m1) (nt1)--(mt1);
        \foreach \i [count=\k,remember=\i as \io (initially 3)] in {4,5}
        {
            \node[vertex] (n\i) at ($(n3)-\k*(dy)$) {};
            \node[vertex] (nt\i) at ($(nt3)-\k*(dy)$) {};
            \draw[edge] (n\i)--(nt\i);
            \ifnum\i>1
                \draw[edge] (n\i)--(n\io) (nt\i)--(nt\io) (nt\i)--(n\io);
            \fi
        }
        \node[vertex] (n6) at ($(n4)-(dyy)$) {};
        \draw[edge] (n6)--(n4) (n6)--(n5);
        \node[vertex] (r1) at ($(n6)+(dx)$) {};
        \node[vertex] (r2) at ($(r1)+(dx)$) {};
        \draw[edge,colX] (n6)--(r1) (r1)--(r2) (r2)--(bt-1);
    \end{tikzpicture}
    \caption{A graph with a flexible matchstick realisation (left) that is not a penny realisation since the two purple vertices are too close and a unique rigid penny realisation (right).}
    \label{fig:pennyGlobRigMatchFlex}
\end{figure}
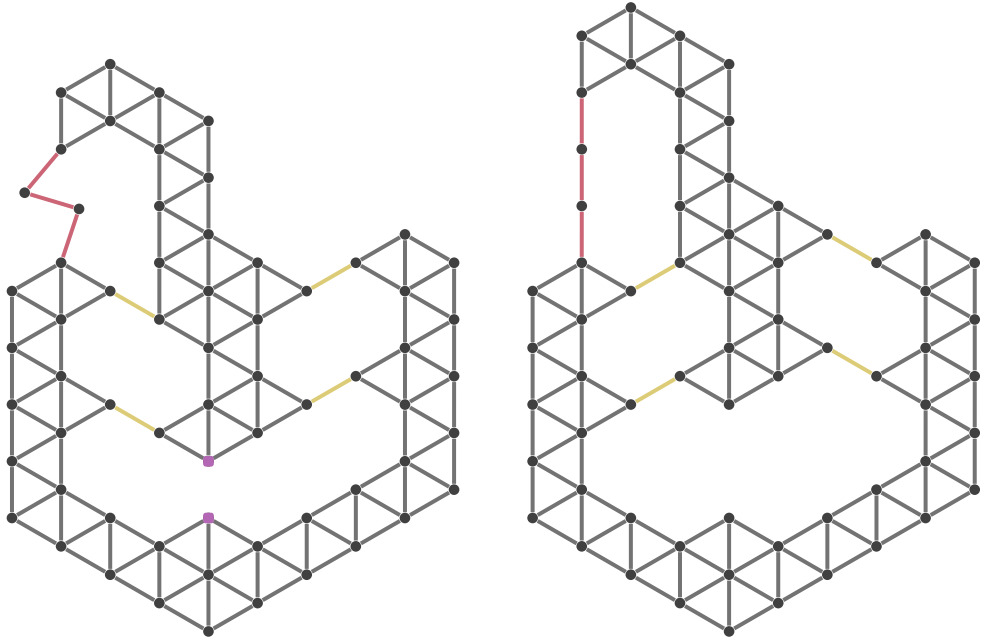

Finally, the graphs in \Cref{fig:pennyMatchRigid} are \urigid, \prigid\, and \mrigid\, but not globally \urigid\, or \mrigid\, since one of the vertices may be flipped to create a new realisation. As opposed to the graph on the left, the one on the right is globally \prigid: due to distance constraints on non-edges, flipping the purple vertex of the graph on the right does not yield a penny framework.

\begin{figure}[ht]
    \centering
    \begin{tikzpicture}[rotate=90,scale=0.75]
        \coordinate (dx) at (0:1);
        \coordinate (dy) at (60:1);
        \coordinate (dyy) at (120:1);
        \clip ($-1*(dy)+(-0.2,-0.2)$) rectangle ($3.5*(dx)+4*(dy)+(0.2,0.2)$);
        \node[vertex] (b0) at (0,0) {};
        \foreach \i [remember=\i as \io (initially 0)] in {1,2,...,4}
        {
            \node[vertex] (b\i) at ($\i*(dy)$) {};
            \node[vertex] (bt\i) at ($-1*(dx)+\i*(dy)$) {};
            \draw[edge] (b\i)--(bt\i) (b\i)--(b\io) (bt\i)--(b\io);
            \ifnum\i>1
                \draw[edge] (bt\i)--(bt\io);
            \fi
        }
        \node[vertex] (c0) at ($5*(dx)$) {};
        \foreach \i [remember=\i as \io (initially 0)] in {1,2,...,4}
        {
            \node[vertex] (c\i) at ($5*(dx)+\i*(dyy)$) {};
            \node[vertex] (ct\i) at ($5*(dx)+1*(dx)+\i*(dyy)$) {};
            \draw[edge] (c\i)--(ct\i) (c\i)--(c\io) (ct\i)--(c\io);
            \ifnum\i>1
                \draw[edge] (ct\i)--(ct\io);
            \fi
        }
        \node[vertex] (t) at ($(b3)+(dx)$) {};
        \draw[edge] (t)--(b3) (t)--(b4) (t)--(c3) (t)--(c4) (b4)--(c4);
        \foreach \i [remember=\i as \io (initially 0)] in {1,2,3,4}
        {
            \node[vertex] (a\i) at ($\i*(dx)$) {};
            \ifnum\i>1
                \draw[edge] (a\i)--(a\io);
            \fi
        }
        \draw[edge] (b0)--(a1) (a4)--(c0) (a1)--(b1) (a4)--(c1);
        \node[vertex] (tt) at ($(a2)-(dyy)$) {};
        \draw[edge] (tt)--(a2) (tt)--(a3);
        \begin{scope}[opacity=0.5]
            \node[vertex] (tt2) at ($(a2)+(dy)$) {};
            \draw[edge,dotted] (tt2)--(a2) (tt2)--(a3);
        \end{scope}
    \end{tikzpicture}
    \quad
    \begin{tikzpicture}[rotate=-90,scale=0.75]
        \coordinate (dx) at (0:1);
        \coordinate (dy) at (60:1);
        \coordinate (dyy) at (120:1);
        \clip ($-1*(dx)-(dy)+(-0.2,-0.2)$) rectangle ($3*(dx)+3*(dy)+(0.2,0.2)$);
        \node[vertex] (b0) at (0,0) {};
        \foreach \i [remember=\i as \io (initially 0)] in {1,2,...,3}
        {
            \node[vertex] (b\i) at ($\i*(dy)$) {};
            \node[vertex] (bt\i) at ($-1*(dx)+\i*(dy)$) {};
            \draw[edge] (b\i)--(bt\i) (b\i)--(b\io) (bt\i)--(b\io);
            \ifnum\i>1
                \draw[edge] (bt\i)--(bt\io);
            \fi
        }
        \node[vertex] (c0) at ($3*(dx)$) {};
        \foreach \i [remember=\i as \io (initially 0)] in {1,2,...,3}
        {
            \node[vertex] (c\i) at ($3*(dx)+\i*(dyy)$) {};
            \node[vertex] (ct\i) at ($3*(dx)+1*(dx)+\i*(dyy)$) {};
            \draw[edge] (c\i)--(ct\i) (c\i)--(c\io) (ct\i)--(c\io);
            \ifnum\i>1
                \draw[edge] (ct\i)--(ct\io);
            \fi
        }
        \draw[edge] (b2)--(c2);
        \node[vertex] (a1) at (dx) {};
        \node[vertex] (a2) at ($2*(dx)$) {};
        \draw[edge] (a1)--(a2) (a1)--(b0) (a1)--(b1) (a2)--(c0) (a2)--(c1);
        \node[hvertex] (tt) at ($(a1)-(dyy)$) {};
        \draw[edge] (tt)--(a1) (tt)--(a2);
        \begin{scope}[opacity=0.25]
            \node[vertex] (tt2) at ($(a1)+(dy)$) {};
            \draw[edge,dotted] (tt2)--(a1) (tt2)--(a2);
        \end{scope}
    \end{tikzpicture}
    \caption{Graphs which are \urigid/\mrigid, but not globally. Both are \prigid, but only the right one is globally \prigid.}
    \label{fig:pennyMatchRigid}
\end{figure}
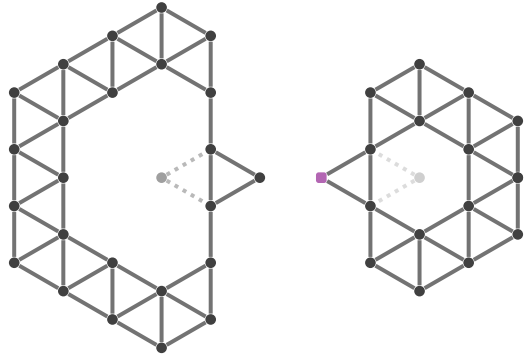

We conclude the section by establishing analogous connections between matchstick rigidity and $2$-unit-distance rigidity properties. See \Cref{tab:comparisontab} (blue) for a summary. 

First, we note that a not \urigid\, graph may be \mrigid, and even globally \mrigid, as shown in \Cref{fig:matchRigUnitFlex}. This graph is not \urigid, since flipping the yellow triangle along the vertical edge and slightly rotating the blue subgraph so that we obtain a non-injective realisation gives a  flexible 2-unit-distance framework.
However, this realisation is not a matchstick realisation. Indeed, the only matchstick realisation of this graph is the one given in the figure (and so the graph is globally \mrigid). If we remove the purple vertices and edges, we obtain  additional (but finitely many) realisations by flipping the green subgraphs (and so the graph remains \mrigid, but is neither \urigid\, nor globally \mrigid).
    
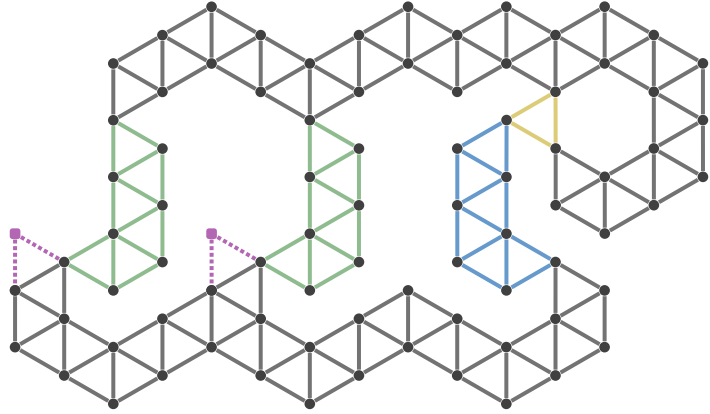
\begin{figure}[ht]
    \centering    
    \begin{tikzpicture}[scale=0.75,rotate=-30]
        \coordinate (dx) at (0:1);
        \coordinate (dy) at (60:1);
        \coordinate (dyy) at (120:1);
        \node[vertex] (b1) at ($(dx)$) {};
        \node[vertex] (bt1) at ($(dx)+(dyy)$) {};
        \draw[edge] (b1)--(bt1);
        \foreach \j [evaluate=\j as \jp using int(\j+1),evaluate=\j as \jpp using int(\j+2)] in {1,3,...,11}
        {
            \foreach \i [
                count=\k,
                remember=\i as \io (initially \j),
                evaluate=\k as \kx using {int(mod(int((\j+1)/2),2)*\k)},
                evaluate=\k as \ky using {int(mod(int((\j+3)/2),2)*\k)},
                evaluate=\k as \kk using {int(mod(int((\j+1)/2),2))},
            ] in {\jp,\jpp}
            {
                \node[vertex] (b\i) at ($(b\j)+\kx*(dx)+\ky*(dy)$) {};
                \node[vertex] (bt\i) at ($(bt\j)+\kx*(dx)+\ky*(dy)$) {};
                \draw[edge] (b\i)--(bt\i) (b\i)--(b\io) (bt\i)--(bt\io);
                \ifnum\kk=0
                    \draw[edge] (b\i)--(bt\io);
                \else
                    \draw[edge] (b\io)--(bt\i);
                \fi
            }
        }
        \foreach \p [count=\i,evaluate=\p as \q using int(\p+1)] in {1,5}
        {
            \node[vertex] (s\i) at ($(bt\p)+(dy)$) {};
            \draw[edge] (s\i)--(bt\p) (s\i)--(bt\q);
        }
        \node[vertex] (s3) at ($(bt12)+(dyy)$) {};
        \draw[edge] (s3)--(bt12) (s3)--(bt13);
        
        \node[vertex] (c1) at ($3*(dy)+2*(dyy)$) {};
        \node[vertex] (ct1) at ($3*(dy)+3*(dyy)$) {};
        \draw[edge] (c1)--(ct1);
        \foreach \j [evaluate=\j as \jp using int(\j+1),evaluate=\j as \jpp using int(\j+2)] in {1,3,...,5}
        {
            \foreach \i [
                count=\k,
                remember=\i as \io (initially \j),
                evaluate=\k as \kx using {int(mod(int((\j+3)/2),2)*\k)},
                evaluate=\k as \ky using {int(mod(int((\j+1)/2),2)*\k)},
                evaluate=\k as \kk using {int(mod(int((\j+3)/2),2))},
            ] in {\jp,\jpp}
            {
                \node[vertex] (c\i) at ($(c\j)+\kx*(dx)+\ky*(dy)$) {};
                \node[vertex] (ct\i) at ($(ct\j)+\kx*(dx)+\ky*(dy)$) {};
                \draw[edge] (c\i)--(ct\i) (c\i)--(c\io) (ct\i)--(ct\io);
                \ifnum\kk=0
                    \draw[edge] (c\i)--(ct\io);
                \else
                    \draw[edge] (c\io)--(ct\i);
                \fi
            }
        }
        \foreach \j [
            evaluate=\j as \i using int(\j+1),
            evaluate=\j as \kx using {int(mod(\j,2))},
            evaluate=\j as \ky using {int(mod(\j+1,2))},
            remember=\i as \io (initially \j),
        ] in {7,8,9,10}
        {
            \node[vertex] (c\i) at ($(c\j)+\kx*(dx)+\ky*(dy)$) {};
            \node[vertex] (ct\i) at ($(ct\j)+\kx*(dx)+\ky*(dy)$) {};
            \draw[edge] (c\i)--(ct\i) (c\i)--(c\io) (ct\i)--(ct\io);
            \ifnum\kx=0
                \draw[edge] (c\i)--(ct\io);
            \else
                \draw[edge] (c\io)--(ct\i);
            \fi
        }
        \foreach \i [count=\k,remember=\i as \io (initially 11)] in {12,13}
        {
            \node[vertex] (c\i) at ($(c11)+\k*(dx)$) {};
            \node[vertex] (ct\i) at ($(ct11)+\k*(dx)$) {};
            \draw[edge] (c\i)--(ct\i) (c\i)--(c\io) (c\io)--(ct\i) (ct\i)--(ct\io);
        }
        \node[vertex] (cc1) at ($(c12)-(dyy)$) {};
        \draw[edge] (cc1)--(c12) (cc1)--(c13);
        \node[vertex] (cc2) at ($(c13)-(dyy)$) {};
        \draw[edge] (cc2)--(c13) (cc2)--(cc1);
        \node[vertex] (cc3) at ($(cc1)-(dyy)$) {};
        \draw[edge] (cc3)--(cc1) (cc3)--(cc2);
        \node[vertex] (cc4) at ($(cc1)-(dy)$) {};
        \draw[edge] (cc4)--(cc1) (cc4)--(cc3);
        \node[vertex] (cc5) at ($(cc3)-(dy)$) {};
        \draw[edge] (cc5)--(cc3) (cc5)--(cc4);
        \node[vertex] (cc6) at ($(cc4)-(dy)$) {};
        \draw[edge] (cc6)--(cc4) (cc6)--(cc5);
        \node[vertex] (cc7) at ($(cc4)-(dx)$) {};
        \draw[edge] (cc7)--(cc4) (cc7)--(cc6);

        \foreach \e [count=\k] in {c1,c5}
        {
            \foreach \i [remember=\i as \io] in {1,2,3}
            {
                \node[vertex] (m\k\i) at ($(s\k)+(dx)+(dy)+\i*(dyy)-(dyy)$) {};
                \node[vertex] (mt\k\i) at ($(s\k)+(dy)+\i*(dyy)-2*(dyy)$) {};
                \draw[edge,colG] (m\k\i)--(mt\k\i);
                \ifnum\i>1
                    \draw[edge,colG] (m\k\i)--(m\k\io) (mt\k\i)--(mt\k\io) (mt\k\i)--(m\k\io);
                \fi
            }
            \draw[edge,colG] (s\k)--(mt\k1) (s\k)--(mt\k2) (\e)--(mt\k3) (\e)--(m\k3);
        }
        \node[vertex] (e) at ($(cc7)-(dx)$) {};
        \foreach \i [remember=\i as \io] in {1,2,3}
        {
            \node[vertex] (m3\i) at ($(s3)-(dx)-(dy)+\i*(dyy)-(dyy)$) {};
            \node[vertex] (mt3\i) at ($(s3)-(dy)+\i*(dyy)-(dyy)$) {};
            \draw[edge,colB] (m3\i)--(mt3\i);
            \ifnum\i>1
                \draw[edge,colB] (m3\i)--(m3\io) (mt3\i)--(mt3\io) (mt3\i)--(m3\io);
            \fi
        }
        \draw[edge,colB] (s3)--(mt31) (s3)--(mt32) (e)--(mt33) (e)--(m33);

        \draw[edge,colO] (e)--(c10) (e)--(cc7) (c10)--(cc7);
        \foreach \i in {1,2}
        {
            \node[hvertex] (t) at ($(s\i)-(dx)$) {};
            \node[vertex] (u) at ($(t)-(dyy)$) {};
            \draw[edge,colP, densely dotted] (t)--(s\i) (t)--(u);
        }
    \end{tikzpicture}
    \caption{A globally \mrigid, not \urigid\, graph. Removing the purple vertices and edges yields a \mrigid, not \urigid\, nor globally \mrigid\, graph.}
    \label{fig:matchRigUnitFlex}
\end{figure}

Finally, it is easy to see that the graph shown in \Cref{fig:matchGlobRigUnitRig} is \urigid\, and globally \mrigid, but not globally \urigid.

\begin{figure}[ht]
    \centering
    \begin{tikzpicture}[scale=0.75]
        \coordinate (dx) at (0:1);
        \coordinate (dy) at (60:1);
        \foreach \i [remember=\i as \io (initially 0)] in {1,2,...,5}
        {
            \node[vertex] (b\i) at ($\i*(dx)$) {};
            \node[vertex] (bt\i) at ($\i*(dx)+(dy)$) {};
            \draw[edge] (b\i)--(bt\i);
            \ifnum\i>1
                \draw[edge] (b\i)--(b\io) (bt\i)--(bt\io) (bt\io)--(b\i);
            \fi
        }
        \begin{scope}
            \coordinate (ex) at (90:1);
            \coordinate (ey) at (150:1);
            \foreach \i [remember=\i as \io (initially 0)] in {0,1,...,3}
            {
                \node[vertex] (c\i) at ($(bt1)+\i*(ex)$) {};
                \node[vertex] (ct\i) at ($(bt1)+\i*(ex)+(ey)$) {};
                \draw[edge] (c\i)--(ct\i);
                \ifnum\i>0
                    \draw[edge] (c\i)--(c\io) (ct\i)--(ct\io) (ct\io)--(c\i);
                \fi
            }
        \end{scope}
        \begin{scope}
            \coordinate (fx) at ($(bt5)!0.2!(c3)-(bt5)$);
            \coordinate[rotate around={-60:(fx)}] (fys) at ($2*(fx)$);
            \coordinate (fy) at ($(fys)-(fx)$);
            \foreach \i [remember=\i as \io (initially 0)] in {-1,0,...,4}
            {
                \node[vertex] (d\i) at ($(bt5)+\i*(fx)+(fy)$) {};
                \node[vertex] (dt\i) at ($(bt5)+\i*(fx)+2*(fy)$) {};
                \draw[edge] (d\i)--(dt\i);
                \ifnum\i>-1
                    \draw[edge] (d\i)--(d\io) (dt\i)--(dt\io) (dt\io)--(d\i);
                \fi
            }
            \node[vertex] (d5) at ($(bt5)+5*(fx)+(fy)$) {};
            \draw[edge] (d5)--(d4) (d5)--(dt4);
            \draw[edge] (d4)--(c3) (d5)--(c3) (d-1)--(bt5) (d0)--(bt5);
        \end{scope}
        \node[hvertex] at (b1) {};
        \node[hvertex] at (ct3) {};
    \end{tikzpicture}
    \caption{A \urigid, globally \mrigid, not globally \urigid\, graph.}
    \label{fig:matchGlobRigUnitRig}
\end{figure}

\section{Families of rigid penny graphs}
\label{sec:thms}
In this section, we study two classes of graphs: outerplanar graphs and graphs with many 3- and 4-cycles. In both cases, we provide equivalent conditions to determine whether a graph is \prigid/\mrigid/\urigid.

\begin{lemma}
    \label{lem:pum-rigid-iff-all-rigid}
    A penny/matchstick/$2$-unit-distance graph $G$ is \prigid/\mrigid/\urigid\ if and only if
    every penny/matchstick/$2$-unit-distance realisation of $G$ is rigid.
\end{lemma}
\begin{proof}
    Suppose that a $2$-unit-distance framework $(G, \rho)$ has a flex $t \mapsto m_t$, $t\in [0,1]$.
    Then $m_t$ is injective for all $t$ in some interval $[0,\varepsilon)$ and so $G$ is not \urigid.
    If $\rho$ is matchstick, then $(G,m_t)$ is planar for all $t$ in some interval $[0,\varepsilon)$ and so $G$ is not \mrigid.
    Finally, suppose that $\rho$ is a penny realisation.
    For $uv\notin E$, let $\varepsilon_{uv}$ be the smallest~$t$ such that $\|m_t(u)-m_t(v)\|=1$,
    or $\varepsilon_{uv}=1$ if no such $t$ exists.
    Since $m_0=\rho$ is a penny realisation, $\varepsilon_{uv}>0$ for all $uv\not\in E$.
    Since there are only finitely many non-edges, $\varepsilon=\min\sst{\varepsilon_{uv}}{uv\notin E}$ is positive
    and hence all $m_t$ for $t\in[0,\varepsilon)$ are penny realisations of $G$, i.e., $(G, \rho)$ is not \prigid.

    The other implication follows from the fact that
    a framework can have only finitely many non-congruent rigid frameworks equivalent to it.
\end{proof}

\begin{proposition}\label{prop:pennyouterplanar}
    Let $G=(V,E)$ be an outerplanar penny/matchstick/$2$-unit-distance graph with $|V|\geq3$.
    Then every penny/matchstick/2-unit-distance framework $(G,\rho)$ is independent.
    Furthermore,
    the following are equivalent:
    \begin{enumerate}[noitemsep]
        \item\label{it:pennyouterplanar:prigid} $G$ is \prigid/\mrigid/\urigid;
        \item\label{it:pennyouterplanar:gprigid} $G$ is globally \prigid/\mrigid/\urigid;
        \item\label{it:pennyouterplanar:maxout} $G$ is a maximal outerplanar graph;
        \item\label{it:pennyouterplanar:edgecount} $|E|=2|V|-3$;
        \item\label{it:pennyouterplanar:2tree} $G$ is a 2-tree.
    \end{enumerate}
\end{proposition}

\begin{proof}
    We prove the result for penny graphs. The same proof applies to matchstick and 2-unit-distance graphs. We start by proving the independence part of the result by induction on $|V|$.
    The statement is clearly true for $|V|=3$, since the complete graph $K_3$ has only one realisation, and triangles are independent.
    Now suppose that all outerplanar penny graphs on at most $n-1$ vertices have only independent penny frameworks, and let $|V|=n$.
    Choose a penny framework $(G,\rho)$.
    If $G$ is not 2-connected,
    it has a cut vertex.
    It follows that $(G,\rho)$ is formed from gluing two outerplanar penny frameworks at a vertex, and thus $(G,\rho)$ is independent by the induction hypothesis.
    Suppose instead that $G$ is 2-connected.
    Then $G$ has a (unique) Hamiltonian cycle, any chord of which is a cut set.
    If $G$ has no chords then $(G,\rho)$ is independent as any framework of a cycle that is not restricted to a line is independent.
    If $G$ has a chord, then $G$ is formed from gluing two outerplanar penny frameworks at an edge,
    and hence $(G,\rho)$ is independent by the induction hypothesis.

    We now prove the second part of the statement.
    It is well-known that \ref{it:pennyouterplanar:maxout}, \ref{it:pennyouterplanar:edgecount} and \ref{it:pennyouterplanar:2tree} are equivalent.
    Let $(G,\rho)$ be a penny realisation.
    Since it is independent,
    it is also regular, i.e., the rank of $R(G,\rho)$ is the maximum possible among all $(G,\rho')$.
    By the main result of~\cite{AsimowRoth1978},
    $(G,\rho)$ is rigid if and only if $|E|=2|V|-3$.
    Hence, \ref{it:pennyouterplanar:edgecount} and \ref{it:pennyouterplanar:prigid} are equivalent
    by \Cref{lem:pum-rigid-iff-all-rigid}.
    It is immediate that \ref{it:pennyouterplanar:gprigid} implies \ref{it:pennyouterplanar:prigid}.
    Finally \ref{it:pennyouterplanar:2tree} implies \ref{it:pennyouterplanar:gprigid} by \Cref{lem:d_sphere_unique_realisation}.
\end{proof}

We conclude this section by considering special edge colourings
that characterise the existence of a flexible realisation in $\mathbb{R}^2$ for general graphs; 
see~\cite{GraseggerLegerskySchicho2019} for more details.
\begin{definition}
    Let $G=(V,E)$ be a graph.
    A colouring of edges $\delta\colon  E\rightarrow \{\text{\blue{}, \red{}}\}$
    is called a \emph{NAC-colouring},
    if it is surjective and for every cycle in~$G$,
    either all edges have the same colour, or
    there are at least two edges of each colour.
    A NAC-colouring is called \emph{Cartesian} if no pair of vertices is connected by a red path and blue path simultaneously.
\end{definition}

We connect NAC-colourings with rigidity properties in \Cref{thm:allFaces3or4cycles}. To do so, we use the notion of `TP-framework' and a result from \cite{GraseggerLegersky2024}.
\begin{definition}
    Let $G=(V,E)$ be a graph.
    Let $\asc{G}$ be the 2-dimensional abstract simplicial complex
    obtained from $V\cup E$ by adding $\{u,v\}$ whenever $u,v\in V$ lie in an induced 4-cycle
    and $\{u,v,w\}$ whenever $u,v,w\in V$ lie in a 3-cycle or an induced 4-cycle.
    Let $\rho$ be a realisation of $G$ such that every 3-cycle is mapped to a non-degenerate triangle and every 4-cycle is mapped to
    a non-degenerate parallelogram.\footnote{By ``degenerate'' we mean all joints being collinear.}
    If $\asc{G}$ is simply connected, the framework $(G, \rho)$ is called a \emph{TP-framework}.
\end{definition}

\begin{theorem}[{\cite[Thm~1.24]{GraseggerLegersky2024}}]
    \label{thm:GraseggerLegersky2024}
    Let $(G,\rho)$ be a TP-framework.
    The following statements are equivalent:
    \begin{enumerate}[noitemsep]
        \item $(G,\rho)$ is rigid,
        \item $(G,\rho)$ is infinitesimally rigid,
        \item $G$ has no Cartesian NAC-colouring.
    \end{enumerate}
\end{theorem}

\begin{theorem}
    \label{thm:allFaces3or4cycles}
    Let $G$ be a graph that has a penny/matchstick realisation $\rho$ such that all bounded faces are 3- or 4-cycles. 
    Then the following statements are equivalent:
    \begin{enumerate}[noitemsep]
        \item\label{it:pennyRigid} $G$ is \prigid/\mrigid,
        \item\label{it:2unitDistRigid} $G$ is \urigid,
        \item\label{it:contRigid} $(G,\rho)$ is rigid,
        \item\label{it:infRigid} $(G,\rho)$ is infinitesimally rigid,
        \item\label{it:noNAC} $G$ has no Cartesian NAC-colouring.
    \end{enumerate}
    In particular, if $G$ is a penny graph,
    then $G$ is \prigid\, if and only if it is \mrigid.
\end{theorem}
\begin{proof}
    The equivalence of \ref{it:contRigid}--\ref{it:noNAC} follows from \Cref{thm:GraseggerLegersky2024} once we show
    that $(G,\rho)$ is a TP-framework.
    Indeed, since $\rho$ is a penny/matchstick-realisation, it is injective and all 3-cycles are equilateral, hence non-degenerate, triangles;
    each 4-cycle is a rhombus with no coincident vertices, and so a non-degenerate parallelogram. 
    Since $\rho$ is planar and all bounded faces are triangles or
    rhombi by the assumption, $\asc{G}$ is simply connected. 

    By the argument, any penny/matchstick/$2$-unit-distance framework $(G,\rho)$ is a TP-framework.
    Since \ref{it:noNAC} is independent of $\rho$,
    either all penny/matchstick/$2$-unit-distance realisations of $G$ are rigid,
    or none of them is by \Cref{thm:GraseggerLegersky2024}.
    Therefore, \ref{it:pennyRigid}, \ref{it:2unitDistRigid} and \ref{it:noNAC} are equivalent
    by \Cref{lem:pum-rigid-iff-all-rigid}.
\end{proof}

The existence of Cartesian NAC-colourings
can be decided in polynomial time in the situation of \Cref{thm:allFaces3or4cycles} using results from~\cite{GraseggerLegersky2024}.

From now on colours in figures show NAC-colourings.
Without the assumption that the bounded faces are 3- or 4-cycles, \Cref{thm:allFaces3or4cycles} does not hold, as shown in \Cref{fig:pennyCartNAC} (right) and \Cref{fig:genFlexibleNoCartNAC}.

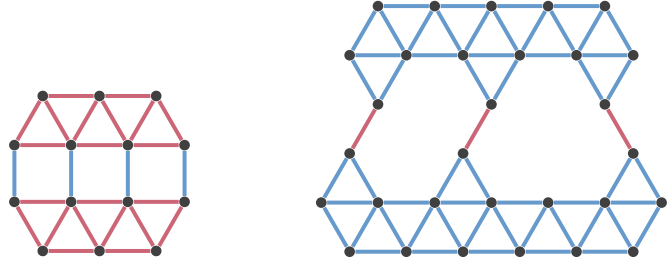
\begin{figure}[ht]
    \centering
    \begin{tikzpicture}[scale=0.75]
        \pgfmathparse{0}\let\w=\pgfmathresult
        \foreach \x in {1,2,3,4}
        {
            \node[vertex] (b\x) at ($(\x,0)+(90-\w:1)$) {};
            \node[vertex] (bp\x) at (\x,0) {};
        }
        \foreach \x [evaluate=\x as \nx using int(\x+1)] in {1,2,3}
        {
            \node[vertex,rotate around={60:(b\x)}] (c\x) at (b\nx) {};
            \node[vertex,rotate around={-60:(bp\x)}] (cp\x) at (bp\nx) {};
            \draw[edge,colR] (b\x)edge(b\nx);
            \draw[edge,colR] (b\x)edge(c\x);
            \draw[edge,colR] (b\nx)edge(c\x);
            \draw[edge,colR] (bp\x)edge(bp\nx);
            \draw[edge,colR] (bp\x)edge(cp\x);
            \draw[edge,colR] (bp\nx)edge(cp\x);
        }
        \draw[edge,colR] (c1)edge(c2) (c2)edge(c3) (cp1)edge(cp2) (cp2)edge(cp3);
        \draw[edge,colB] (b1)edge(bp1) (b2)edge(bp2) (b3)edge(bp3) (b4)edge(bp4);
    \end{tikzpicture}
    \qquad\qquad
    \begin{tikzpicture}[scale=0.75]
        \coordinate (dx) at (1,0);
        \coordinate (dy) at ($(0,{sqrt(3)/2})$);
        \node[vertex] (1) at ($0*(dx)+0*(dy)$) {};
        \node[vertex] (2) at ($1*(dx)+0*(dy)$) {};
        \node[vertex] (3) at ($2*(dx)+0*(dy)$) {};
        \node[vertex] (4) at ($3*(dx)+0*(dy)$) {};
        \node[vertex] (5) at ($4*(dx)+0*(dy)$) {};
        \node[vertex] (6) at ($5*(dx)+0*(dy)$) {};
        \node[vertex] (7) at ($-1/2*(dx)+1*(dy)$) {};
        \node[vertex] (8) at ($1/2*(dx)+1*(dy)$) {};
        \node[vertex] (9) at ($3/2*(dx)+1*(dy)$) {};
        \node[vertex] (10) at ($5/2*(dx)+1*(dy)$) {};
        \node[vertex] (11) at ($7/2*(dx)+1*(dy)$) {};
        \node[vertex] (12) at ($9/2*(dx)+1*(dy)$) {};
        \node[vertex] (13) at ($11/2*(dx)+1*(dy)$) {};
        \node[vertex] (14) at ($5*(dx)+2*(dy)$) {};
        \node[vertex] (15) at ($0*(dx)+2*(dy)$) {};
        \node[vertex] (16) at ($2*(dx)+2*(dy)$) {};
        \node[vertex] (17) at ($0*(dx)+4*(dy)$) {};
        \node[vertex] (18) at ($1*(dx)+4*(dy)$) {};
        \node[vertex] (19) at ($2*(dx)+4*(dy)$) {};
        \node[vertex] (20) at ($3*(dx)+4*(dy)$) {};
        \node[vertex] (21) at ($4*(dx)+4*(dy)$) {};
        \node[vertex] (22) at ($5*(dx)+4*(dy)$) {};
        \node[vertex] (23) at ($1/2*(dx)+5*(dy)$) {};
        \node[vertex] (24) at ($3/2*(dx)+5*(dy)$) {};
        \node[vertex] (25) at ($5/2*(dx)+5*(dy)$) {};
        \node[vertex] (26) at ($7/2*(dx)+5*(dy)$) {};
        \node[vertex] (27) at ($9/2*(dx)+5*(dy)$) {};
        \node[vertex] (28) at ($1/2*(dx)+3*(dy)$) {};
        \node[vertex] (29) at ($5/2*(dx)+3*(dy)$) {};
        \node[vertex] (30) at ($9/2*(dx)+3*(dy)$) {};
        \draw[edge,colB] (1)--(2) (1)--(7) (1)--(8) (2)--(3) (2)--(8) (2)--(9) (3)--(4) (3)--(9) (3)--(10) (4)--(5) (4)--(10) (4)--(11) (5)--(6) (5)--(11) (5)--(12) (6)--(12) (6)--(13) (7)--(8) (7)--(15) (8)--(9) (8)--(15) (9)--(10) (9)--(16) (10)--(11) (10)--(16) (11)--(12) (12)--(13) (12)--(14) (13)--(14);
        \draw[edge,colR] (14)--(30) (15)--(28) (16)--(29);
        \draw[edge,colB] (17)--(18) (17)--(23) (17)--(28) (18)--(19) (18)--(23) (18)--(24) (18)--(28) (19)--(20) (19)--(24) (19)--(25) (19)--(29) (20)--(21) (20)--(25) (20)--(26) (20)--(29) (21)--(22) (21)--(26) (21)--(27) (21)--(30) (22)--(27) (22)--(30) (23)--(24) (24)--(25) (25)--(26) (26)--(27);
    \end{tikzpicture}
    \caption{Two penny realisations of graphs with Cartesian NAC-colourings.
    Since all bounded faces of the left one are 3- or 4-cycles it is not \prigid\ by \Cref{thm:allFaces3or4cycles}. The right one, however, is \prigid.}
    \label{fig:pennyCartNAC}
\end{figure}

\begin{figure}[ht]
    \centering
    \begin{tikzpicture}[scale=0.75]
        \coordinate (dx) at (0:1);
        \coordinate (dy) at (60:1);
        \node[vertex] (bt0) at (dy) {};
        \foreach \i [remember=\i as \io (initially 0)] in {1,2,...,10}
        {
            \node[vertex] (b\i) at ($\i*(dx)$) {};
            \node[vertex] (bt\i) at ($\i*(dx)+(dy)$) {};
            \draw[edge,colB] (b\i)--(bt\i) (bt\i)--(bt\io) (bt\io)--(b\i);
            \ifnum\i>1
                \draw[edge,colB] (b\i)--(b\io);
            \fi
        }
        \foreach \x [count=\i,evaluate=\x as \xm using int(\x-1)] in {1,4,7,10}
        {
            \node[vertex] (c\i) at ($\x*(dx)-(dx)+2*(dy)$) {};
            \draw[edge,colB] (c\i)--(bt\x) (c\i)--(bt\xm);
        }
        \begin{scope}
            \foreach \w in {25}
            {
                \coordinate (dxm) at (-\w:1);
                \coordinate (dym) at (60-\w:1);
                \foreach \i [remember=\i as \io] in {0,1,...,4}
                {
                    \node[vertex] (m\i) at ($(c2)-\i*(dxm)+(dym)$) {};
                    \node[vertex] (mt\i) at ($(c2)-\i*(dxm)+2*(dym)$) {};
                    \draw[edge,colR] (m\i)--(mt\i);
                    \ifnum\i>0
                        \draw[edge,colR] (m\i)--(m\io) (mt\i)--(mt\io) (mt\i)--(m\io);
                    \fi
                }
                \draw[edge,colR] (c2)--(m0) (c2)--(m1);
                \node[vertex] (n2c1) at ($(c2)-(dxm)+3*(dym)$) {};
                \draw[edge,colR] (n2c1)--(mt0) (n2c1)--(mt1);
                \node[vertex] (mdc1) at ($(c2)-3*(dxm)$) {};
                \draw[edge,colR] (mdc1)--(m4) (mdc1)--(m3);
                \foreach \k in {3,4}
                {
                    \foreach \i [remember=\i as \io] in {1,2}
                    {
                        \node[vertex] (n\k\i) at ($(c\k)+\i*(dym)$) {};
                        \node[vertex] (nt\k\i) at ($(c\k)-(dxm)+\i*(dym)$) {};
                        \draw[edge,colR] (n\k\i)--(nt\k\i);
                        \ifnum\i>1
                            \draw[edge,colR] (n\k\i)--(n\k\io) (nt\k\i)--(nt\k\io) (nt\k\i)--(n\k\io);
                        \fi
                    }
                    \draw[edge,colR] (c\k)--(n\k1) (c\k)--(nt\k1);
                    \node[vertex] (n\k c1) at ($(c\k)-(dxm)+3*(dym)$) {};
                    \draw[edge,colR] (n\k c1)--(n\k2) (n\k c1)--(nt\k2);
                }
            }
        \end{scope}
        \begin{scope}
            \node[vertex] (db0) at ($(n4c1)+(dy)$) {};
            \foreach \i [remember=\i as \io (initially 0)] in {1,2,...,7}
            {
                \node[vertex] (d\i) at ($(n4c1)-\i*(dx)+2*(dy)$) {};
                \node[vertex] (db\i) at ($(n4c1)-\i*(dx)+(dy)$) {};
                \draw[edge,colB] (d\i)--(db\i) (db\i)--(db\io) (db\io)--(d\i);
                \ifnum\i>1
                    \draw[edge,colB] (d\i)--(d\io);
                \fi
            }
            \foreach \x [count=\i,evaluate=\i as \ii using int(5-\i),evaluate=\x as \xm using int(\x-1)] in {1,4,7}
            {
                \draw[edge,colB] (n\ii c1)--(db\x) (n\ii c1)--(db\xm);
            }
        \end{scope}
        \begin{scope}
            \path[name path=circ1] (c1) circle[radius=1cm];
            \path[name path=circ2] (mdc1) circle[radius=1cm];
            \path[name intersections={of=circ1 and circ2,name=p}];
            \node[vertex] (p1) at (p-1) {};
            \node[vertex] (p2) at (p-2) {};
            \draw[edge,colB] (c1)--(p1) (mdc1)--(p2);
            \draw[edge,colR] (c1)--(p2) (mdc1)--(p1);
        \end{scope}

        \begin{scope}
        \begin{pgfonlayer}{background}
            \pgfmathparse{0.3}\let\dd=\pgfmathresult
            \fill[black!10!white,rounded corners=2pt] 
                ($(b1)+(240:\dd)$) -- 
                ($(b4)+(-60:\dd)$) -- 
                ($(bt4)+(0:\dd)$) --                 
                ($(c2)+(-12.5:\dd)$) -- 
                ($(m0)+(-25:\dd)$) -- 
                ($(mt1)+(35:\dd)$) --
                ($(mt4)+(95:\dd)$) -- 
                ($(m4)+(155:\dd)$) -- 
                ($(mdc1)+(156.5:\dd)$) -- 
                ($(p2)+(167.5:\dd+0.05)$) -- 
                ($(c1)+(183.5:\dd)$) -- 
                ($(bt0)+(180:\dd)$) -- cycle;
        \end{pgfonlayer}
        \end{scope}
    \end{tikzpicture}
    \caption{Non-\prigid\, graphs with NAC-colourings, but no Cartesian one.
        The subgraph indicated by the gray area is generically flexible whilst the whole graph is generically rigid.}
    \label{fig:genFlexibleNoCartNAC}
\end{figure}
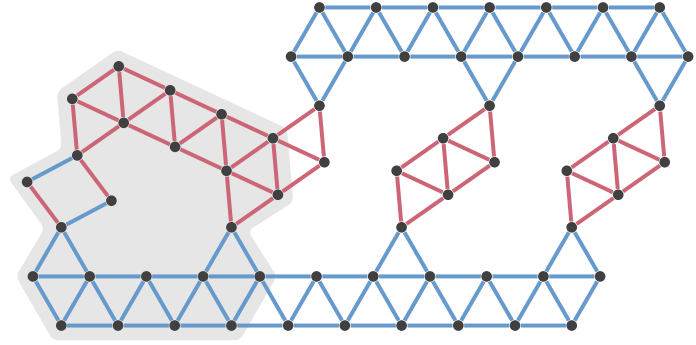

\phantomsection
\addcontentsline{toc}{section}{Acknowledgements}
\section*{Acknowledgements}
This research was funded in part by the Austrian Science Fund (FWF)
10.55776/I6233. For open access purposes, the authors have applied a CC BY public copyright license to any author accepted manuscript version arising from this submission.

This project originated from the Fields Institute Focus Program on Geometric Constraint Systems.
The authors are grateful to the Fields Institute for their hospitality and financial support.

S.\,D.\ was supported by the Heilbronn Institute for Mathematical Research.
G.\,G.\ was supported by the Austrian Science Fund (FWF) 10.55776/I6233.
J.\,L.\ was funded by the Czech Science Foundation (GAČR) project 22-04381L.
A.\,N.\ was partially supported by EPSRC grant number EP/X036723/1.

\phantomsection

\end{document}